\documentclass[reqno, 10pt]{amsart}
\usepackage{amssymb}
\usepackage{amsthm}
\usepackage{amsmath}
\usepackage[colorlinks=true,
pdfstartview=FitV, linkcolor=blue, citecolor=green,
urlcolor=]{hyperref}

\usepackage{mathrsfs}

\usepackage[scr=dutchcal]{mathalfa}
\usepackage{accents}
\usepackage{mathtools}
\usepackage{enumerate}

\usepackage{tikz}
\usetikzlibrary{arrows,calc}
\usepackage{tikz-cd}

\usepackage{hyperref}
\usepackage[style=alphabetic,maxcitenames=5,maxbibnames=7]{biblatex}
\usepackage{todonotes}

\theoremstyle{plain}
\newtheorem{proposition}{Proposition}[section]
\newtheorem{theorem}[proposition]{Theorem}
\newtheorem{lemma}[proposition]{Lemma}
\newtheorem{corollary}[proposition]{Corollary}
\newtheorem{observation}[proposition]{Observation}
\newtheorem{question}[proposition]{Question}

\theoremstyle{definition}

\newtheorem{definition}[proposition]{Definition}

\theoremstyle{remark}
\newtheorem{remark}[proposition]{Remark}

\DeclareMathOperator{\Aut}{Aut}

\DeclareMathOperator{\Gr}{\mathsf{Gr}}

\DeclareMathOperator{\SL}{\mathsf{SL}}
\DeclareMathOperator{\GL}{\mathsf{GL}}

\DeclareMathOperator{\PSL}{\mathsf{PSL}}
\DeclareMathOperator{\PGL}{\mathsf{PGL}}
\DeclareMathOperator{\PO}{\mathsf{PO}}
\DeclareMathOperator{\Sp}{\mathsf{Sp}}

\DeclareMathOperator{\Hc}{\mathcal{H}}

\DeclareMathOperator{\Pc}{\mathcal{P}}

\DeclareMathOperator{\Cb}{\mathbb{C}}

\DeclareMathOperator{\Hb}{\mathbb{H}}

\DeclareMathOperator{\Kb}{\mathbb{K}}
\DeclareMathOperator{\Nb}{\mathbb{N}}

\DeclareMathOperator{\Rb}{\mathbb{R}}

\DeclareMathOperator{\Zb}{\mathbb{Z}}

\DeclareMathOperator{\Psf}{\mathsf{P}}

\newcommand{\into}{\hookrightarrow}
\newcommand{\proj}{\mathbb{P}}
\newcommand{\peripherals}{\Pc}

\title{Topological restrictions on relatively Anosov representations}
\author{Konstantinos Tsouvalas and Feng Zhu }
\date{}
\AtEndDocument{\bigskip{\footnotesize%
\textsc{Max Planck Institute for Mathematics in the Sciences, Inselstra{\ss}e 22, 04103 Leipzig, Germany} \par  
 \textit{E-mail address}: \texttt{konstantinos.tsouvalas@mis.mpg.de}\par
 \medskip
\textsc{Department of Mathematics, University of Wisconsin – Madison, WI 53706, USA} \par  
\textit{E-mail address}: \texttt{fzhu52@wisc.edu}}}

\begin{document}
\frenchspacing
\begin{abstract}
We obtain restrictions on which groups can admit relatively Anosov representations into specified target Lie groups, by examining the topology of possible Bowditch boundaries and how they interact with the Anosov limit maps. For instance, we prove that, up to finite index, any group admitting a relatively Anosov representation into $\SL_3(\Rb)$ is a free group or surface group, and any group admitting a relatively $k$-Anosov representation into $\Sp_{2m}(\Rb)$, where $k$ is an odd integer between 1 and $m$, is a surface group or a free product of nilpotent groups. 

We also obtain a characterization of groups admitting relatively 1-Anosov representations into $\SL_4(\Rb)$, general bounds on the dimension of the Bowditch boundary of groups admitting relatively Anosov representations into $\SL_d(\Rb)$, statements relating spheres in the Bowditch boundary to the (non-)existence of relatively Anosov representations, and a characterization of groups of cohomological dimension at least $d-1$ admitting relatively 1-Anosov representations into $\SL_d(\Rb)$.
\end{abstract}

\maketitle

\section{Introduction}

In this paper, we study topological restrictions on relatively hyperbolic groups admitting relatively Anosov
representations into $\SL_d(\Rb)$. 
Relatively Anosov representations, introduced in \cite{KapovichLeeb} and \cite{reldomreps} (see also \cite{ZZ1} and \cite{Tianqi}), provide a higher-rank analogue of geometrically finite representations into rank-one Lie groups.
The construction and theory of these representations is inspired by the theory of Anosov representations, which were introduced by Labourie \cite{Lab} and Guichard--Wienhard \cite{GW} and provide a higher-rank analogue of convex cocompact representations into rank-one Lie groups. 

In many cases, Anosov representations come from convex cocompact actions on properly convex domains in real projective spaces \cite{DGK,Zimmer}. Similarly, given any relatively Anosov representation $\rho\colon (\Gamma,\peripherals) \to \PGL_d(\Rb)$ whose image preserves a properly convex domain $\Omega \subset \proj(\Rb^d)$, there is an action of $\rho(\Gamma)$ on a properly convex domain $\Omega' \subset \proj(\Rb^d)$ which is geometrically finite in an appropriate sense (\cite[Prop.\ 1.13]{ZZ2}, see also \cite{rAGF_Mitul}.)

Many examples of Anosov representations 
come from embeddings of rank-one lattices---or more generally convex cocompact subgroups of rank-one Lie groups---into larger Lie groups and deformations of these embeddings.
There are other sources of Anosov representations, for example Coxeter groups \cite{DGKLM}, fundamental groups of certain Gromov--Thurston manifolds \cite{KapovichGT}, and, more broadly, cubulated hyperbolic groups \cite{DFWZ}.
Nevertheless, most of these other domain groups are only known to admit Anosov representations into Lie groups of high dimension relative to the complexity of the domain group. 

Restrictions on Anosov representations were established in \cite{Tsouvalas,CanaryTsouvalas}. In essence, these results state that restricting the complexity of the domain hyperbolic group $\Gamma$ relative to the complexity of the target Lie group $\mathsf{G}$ constrains what the domain group can be. Here, complexity of the domain can be measured in terms of cohomological dimension, and complexity of the target in terms of $\dim \mathsf{G}$ and/or the size of the parabolic subgroup $\mathsf{P} < \mathsf{G}$ with respect to which we have the Anosov condition. Very informally speaking, the smaller the Lie group and/or the larger the parabolic subgroup, the smaller the resulting flag space $\mathsf{G}/\mathsf{P}$, and the less complexity the limit maps associated to the Anosov representation can exhibit.
When the target group is $\SL_3(\Rb)$, for instance, the only groups admitting Anosov representations are surface groups and free groups. Here and below, by a {\em surface group}, we mean the fundamental group of a closed hyperbolic surface.  When the target group is $\SL_4(\Rb)$, any group admitting a projective Anosov representation is a convex cocompact subgroup of $\PO(1,3)$.

Analogously, many known examples of relatively Anosov representations originate from geometrically finite subgroups of rank-one Lie groups.
There are other known examples, for instance free products of nilpotent factors \cite[Prop.\ 1.19]{ZZ2}. However, as in the Anosov case, all of the known examples where the domain groups are not geometrically finite subgroups of rank-one Lie groups require the target Lie groups to have high dimension.

We can thus ask if the results in \cite{Tsouvalas, CanaryTsouvalas} have relative analogues. In other words, what restrictions are placed on a relatively hyperbolic group if we ask for it to admit a relatively Anosov representation into a fixed semisimple Lie group, with respect to a fixed parabolic subgroup?

Our first main result is the following characterization of the domain groups of relatively Anosov representations into $\mathsf{SL}_3(\mathbb{R})$. 

\begin{theorem}\label{thm:SL3R'} Let $\rho\colon (\Gamma,\mathcal{P}) \to \SL_3(\Rb)$ be a relatively $1$-Anosov representation. Then $\Gamma$ is virtually a free group or a surface group.
\end{theorem}

We also obtain an analogous characterization for the domain groups of $k$-relatively Anosov representations into the symplectic group $\Sp_{2m}(\mathbb{R})$ when $1\leq k \leq m$ is an odd integer.

\begin{theorem}\label{symplectic}
Let $\rho\colon (\Gamma,\mathcal{P}) \to \Sp_{2m}(\Rb)$ be a relatively $k$-Anosov representation, where $1\leq k \leq m$ is an odd integer. Then $\Gamma$ is virtually a free product of nilpotent groups or virtually a surface group.\end{theorem}

We remark that the characterization of domain groups of $k$-Anosov representations into $\Sp_{2m}(\mathbb{R})$ for odd $1\leq k\leq m$ was obtained independently in \cite{DGR}, using results from \cite{Dey-Borel}, and in \cite{PT}. 

Since the image of the exterior power $\wedge^{2q+1}\colon\mathsf{SL}_{4q+2}(\mathbb{R})\rightarrow \mathsf{SL}(\wedge^{2q+1}\mathbb{R}^{4q+2})$ preserves the symplectic form $\omega\colon \wedge^{2q+1}\mathbb{R}^{4q+2}\times \wedge^{2q+1}\mathbb{R}^{4q+2}\rightarrow \mathbb{R}$ defined by $\omega(u,v)=u\wedge v$, as a corollary of Theorem \ref{symplectic} we obtain:

\begin{corollary}\label{(2q+1)'}
Let $\rho\colon (\Gamma,\mathcal{P}) \to \SL_{4q+2}(\Rb)$ be a relatively $(2q+1)$-Anosov representation. Then $\Gamma$ is virtually a free product of nilpotent groups or virtually a surface group.\end{corollary}

These restrictions are topological in nature, in the sense that they come from or are related to the topology of the Bowditch boundary $\partial(\Gamma,\peripherals)$ of the relatively hyperbolic group $(\Gamma,\peripherals)$. Topological features of this boundary are closely related to algebraic properties of the group, but allow us to use arguments involving the geometry, topology, and dynamics of relatively Anosov representations and their limit maps.
Given mild connectivity conditions on $\partial(\Gamma,\peripherals)$, and in the presence of a relatively (1-)Anosov representation $\rho\colon (\Gamma,\peripherals) \to \mathsf{G}$ with limit map $\xi$, the image $\xi(\partial(\Gamma,\peripherals))$ is contained in an affine chart of a projective space associated to $\mathsf{G}$. Moreover, we can find a properly convex domain whose boundary contains $\xi(\partial(\Gamma,\peripherals))$ (see Proposition \ref{prop:rel Anosov preserving domain}.)

One consequence of this, since the limit map $\xi$ is an embedding, is that the Bowditch boundary embeds into a sphere. Thus the study of homeomorphism types of such Bowditch boundaries and groups admitting geometrically finite convergence actions on these spaces is relevant here. 

In the $\SL_4(\Rb)$ case, we can obtain restrictions using results about planar 2-dimensional Bowditch boundaries and a little bit of 3-manifold theory:

\begin{theorem} \label{thm:SL4R} Let $\rho\colon (\Gamma,\peripherals) \to \SL_4(\Rb)$ be a relatively 1-Anosov representation. Then, up to taking finite extensions and finite-index subgroups, $\Gamma$ is a free product of \\
\noindent (1) nilpotent groups, \\
\noindent (2) surface groups, and/or \\
\noindent (3) fundamental groups of compact 3-manifolds 
hyperbolic relative to $\Zb$ and/or $\Zb^2$ subgroups.
% which are hyperbolic relative to free abelian subgroups of rank at most 2.
\end{theorem}

Note that in the case where $\Gamma$ has cohomological dimension at least 3, then we will see, from Theorem~\ref{thm:high cd}, that $\Gamma$ must in fact be the fundamental group of a closed hyperbolizable 3-manifold. This is a particular instance of case (3) of Theorem~\ref{thm:SL4R}.
Without the bound on the cohomological dimension, in contrast to the non-relative case in \cite[\S4]{CanaryTsouvalas}, we are not able to conclude that the 3-manifolds which appear in case (2) are necessarily hyperbolizable: there are many more 3-manifolds with fundamental groups which are hyperbolic relative to free abelian subgroups. 

More generally, we can bound the topological dimension of the Bowditch boundary for any group admitting a relatively $k$-Anosov representation into $\SL_d(\Rb)$ in terms of $k,d\in \mathbb{N}$.

\begin{theorem} [Theorem \ref{upperbound-dim}] \label{upperbound-dim in intro}
Let $d\geq 2$ an integer, $1\leq k\leq \frac{d}{2}$ and $\rho:(\Gamma,\mathcal{P})\rightarrow \mathsf{SL}_d(\mathbb{R})$ be a relatively $k$-Anosov representation. Then we have the upper bound $$\textup{dim}(\partial(\Gamma,\mathcal{P}))\leq d-k-1,$$ unless $(d,k)\in \{(2,1),(4,2),(8,4),(16,8)\}$ and $\partial(\Gamma,\mathcal{P})\cong S^{d-k}$.
\end{theorem}

We are not aware if this bound is optimal in general. We provide the following result about the absence of spheres, of maximum possible dimension, in the Bowditch boundary of groups admitting relatively Anosov representations:

\begin{theorem} \label{thm:excluding spheres}
Let $d,k\in \Nb$ with $d\geq 2k+1$ and $\mathbb{K}=\mathbb{R}$ or $\mathbb{C}$. Let $(\Gamma,\peripherals)$ be a relatively hyperbolic group whose Bowditch boundary $\partial(\Gamma,\peripherals)$ properly contains a sphere of dimension $r_{\Kb}(d-k)-1$ \textup{(}where $r_{\Rb}=1$ and $r_{\Cb}=2$\textup{)}. Then there is no relatively $k$-Anosov representation $\rho \colon(\Gamma,\peripherals) \to \SL_d(\Kb)$.
\end{theorem}

By adapting arguments from the proof of this theorem, we also prove the following result for $k$-Anosov representations into $\mathsf{SL}_{2k+1}(\mathbb{R})$ when $k\in \mathbb{N}$ is odd:

\begin{theorem} [Theorem \ref{embedding-sphere1}] Let $\Delta$ be an one-ended hyperbolic group and $k\in \mathbb{N}$ an odd integer. Suppose that there is a $k$-Anosov representation $\rho\colon \Delta \to \SL_{2k+1}(\Rb)$. Then $\partial_{\infty}\Delta$ embeds in $S^k$.\end{theorem}

Theorem \ref{upperbound-dim in intro} is analogous to Theorem 1.3 in Canary--Tsouvalas \cite{CanaryTsouvalas}. The result in Canary--Tsouvalas is formulated in terms of cohomological dimension, making use of the close relation between the cohomological dimension of a hyperbolic group $\Gamma$ and the topological dimension of its Gromov boundary $\partial_\infty\Gamma$ (due to Bestvina--Mess \cite{Bestvina-Mess}): $\dim(\partial_\infty\Gamma) = \mathrm{cd}(\Gamma)-1$.

In the relative case, Manning--Wang have results partially generalizing this relation: if $\Gamma$ is a torsion-free group which is hyperbolic relative to nilpotent groups (so that $\Gamma$ is of type $F$, by work of Dahmani \cite{Dahmani_finite}), and $\mathrm{cd}(\Gamma) < \mathrm{cd}(\Gamma,\peripherals)$, then $\dim(\partial(\Gamma,\peripherals)) = \mathrm{cd}(\Gamma,\peripherals)-1 = \mathrm{cd}(\Gamma)$ \cite[Thm.\ 5.1]{ManningWang}.
It is unclear whether the first equality holds without the additional assumption that $\mathrm{cd}(\Gamma) < \mathrm{cd}(\Gamma,\peripherals)$, given recent results relating the cohomological dimension to a different compactification of a relatively hyperbolic group~\cite{Fukaya}.
% \footnote{By \cite[Cor.\ B]{Fukaya}, the cohomological dimension of a relatively hyperbolic group is related to the topological dimension of its blown-up corona.

% The blown-up corona is *not* the Bowditch boundary. In the case of a Gromov-hyperbolic group, which is hyperbolic relative to an almost-malnormal family of quasiconvex subgroups, the blown-up corona is homeomorphic to the Gromov boundary \cite[Thm.\ C]{Fukaya}, whereas the Bowditch boundary is obtained from the Gromov boundary by collapsing the limit sets of peripheral subgroups \cite[Thm.\ 1.7]{Fukaya}. There are also examples where the Bowditch boundary has smaller or larger topological dimension than the blown-up coronae~\cite[\S5.2]{Fukaya}.

% This appears to say that \cite[Conj.\ 5.4]{ManningWang}is wrong, since in general the topological dimension of the blown-up corona may not be the topological dimension of the Bowditch boundary (see e.g.\ \cite[Prop.\ 5.2]{Fukaya}).}

In terms of cohomological dimension, we obtain restrictions when $\Gamma$ has high cohomological dimension (relative to the dimension of the target) analogous to \cite[Thm.\ 1.5]{CanaryTsouvalas}. A representation $\rho\colon\Gamma \rightarrow \SL_d(\mathbb{R})$ is called a {\em Benoist representation} if $\rho$ has finite kernel and $\rho(\Gamma)$ preserves and acts properly discontinuously and cocompactly on a strictly convex domain of $\mathbb{P}(\mathbb{R}^d)$. 

\begin{theorem}[Theorem~\ref{thm:high cd}] Let $d\geq 4$ and $(\Gamma, \mathcal{P})$ be a torsion-free relatively hyperbolic group of cohomological dimension at least $d-1$. If $\rho\colon (\Gamma,\peripherals) \to \SL_d(\Rb)$ is a relatively 1-Anosov representation, then $\peripherals = \varnothing$, $\Gamma$ is hyperbolic and $\rho$ is a Benoist representation.\end{theorem}

We recall that Benoist proved in \cite{BenoistI} that a discrete subgroup $\mathsf{\Gamma}$ of $\mathsf{PGL}_d(\mathbb{R})$, $d \geq 3$, which preserves and acts cocompactly on strictly convex domain $\Omega$ of $\mathbb{P}(\mathbb{R}^d)$, is necessarily Gromov hyperbolic and its boundary $\partial\Omega$ is of class $\mathcal{C}^1$. Moreover, by \cite[Prop. 6.1]{GW}, the inclusion representation $\Gamma \xhookrightarrow{}  \mathsf{GL}_d(\mathbb{R})$ is $1$-Anosov.

It is less clear, even in the non-relative case, what happens in lower cohomological dimension. Indeed, there is already a wide variety of hyperbolic groups in cohomological dimension 2, and many of these admit 1-Anosov representations into $\SL_d(\Rb)$ (with $d\in \mathbb{N}$ large) \cite{DFWZ}.

\subsection*{Acknowledgements} 
The authors thank Dick Canary for help with 3-manifold groups, and also the Institut des Hautes \'Etudes Scientifiques, the Institut Henri Poincar\'e, and the \'Ecole Normale Sup\'erieure for their hospitality. 
%{We also thank the referee/s for a thorough reading and comments which have helped improved the paper.} 
FZ was partially supported by Israel Science Foundation grant 737/20 and an AMS-Simons Travel Grant. 
KT would like to thank the Max Planck Institute for Mathematics in the Sciences in Leipzig for providing excellent working conditions.
This project received funding from the European Research Council (ERC) under the European's Union Horizon 2020 research and innovation programme (ERC starting grant DiGGeS, grant agreement No 715982). 

\section{Preliminaries} 

\subsection{Convergence groups and relatively hyperbolic groups} 
We say that a group $\Gamma$ acts on a compact Hausdorff metrizable space $M$ as a \emph{convergence group} if for every sequence $(\gamma_n)_{n\in \mathbb{N}}$ of distinct elements in $\Gamma$ there exists a subsequence $(\gamma_{k_n})_{n\in \mathbb{N}}$ and points $x,y \in M$ such that $(\gamma_{k_n}|_{M \smallsetminus \{y\}})_{n\in \mathbb{N}}$ converges locally uniformly to the constant map $x$. For a general reference on convergence groups, see e.g.\ \cite{Tukia}.

Given $\Gamma$ acting on $M$ as a convergence group, an element $\gamma \in \Gamma$ is called \emph{loxodromic} if it has infinite order and fixes exactly two points in $M$. The \emph{limit set} $\Lambda_\Gamma \subset M$ is the set of points $x \in M$ where there exist $y \in M$ and a sequence $(\gamma_n)_{n \in \mathbb{N}}$ in $\Gamma$ where $(\gamma_n|_{M \smallsetminus \{y\}})_{n\in \mathbb{N}}$ converges locally uniformly to the constant map $x$.
A subgroup $\Gamma_0$ of $\Gamma$ is called {\em elementary} it there exists a finite subset of $M$ with at most $2$ points fixed by $\Gamma_0$. If $\Gamma_0$ is non-elementary then $\Lambda_{\Gamma_0}$ is infinite, perfect and contains loxodromic elements, see \cite[Thm.\ 2S + 2T]{Tukia}.

It is a result of Bowditch that a relatively hyperbolic group $(\Gamma,\peripherals)$ acts as a geometrically finite convergence group on its boundary $\partial(\Gamma,\peripherals)$, with maximal parabolic subgroups given by the conjugates of subgroups in $\peripherals$ \cite[\S6]{Bowditch_relhyp}. Indeed, there is a characterization of relatively hyperbolic groups in terms of geometrically finite convergence group actions \cite{Yaman}.
In particular, when $\Gamma$ is non-elementary, $\partial(\Gamma,\peripherals)$ is uncountable and perfect.
For a general reference on relatively hyperbolic groups, see e.g.\ \cite{Bowditch_relhyp} or \cite{Teddy_EGF}.

\subsection{Convergence groups acting on the circle} 
We have the following theorem about convergence groups acting on the circle.

\begin{theorem}[Gabai \cite{Gabai} +Tukia \cite{Tukia-Fuchsian}] \label{thm:convergence groups on S1 are Fuchsian}
If a group $\Gamma$ acts as a convergence group on $S^1$, then $\Gamma$ is isomorphic to a Fuchsian group.
\end{theorem}

We can combine this with Dunwoody's accessibility theorem to obtain a statement about finite extensions of Fuchsian groups, whose proof we include here for completeness. For the notion of the fundamental group of a finite graph of groups we refer the reader to \cite{Serre-trees}.

\begin{theorem}\label{GabDun} Let $\Gamma'$ be a finitely generated discrete subgroup of $\PSL_2(\Rb)$. Suppose that $1\to F \to \Gamma \to \Gamma'\to 1$ is a short exact sequence of groups where $F$ is finite. Then $\Gamma$ contains a finite-index subgroup isomorphic to a free group or a surface group.\end{theorem}

\begin{proof} Since $\Gamma'$ acts properly discontinuously on the hyperbolic plane $\Hb^2$, it is virtually a free or a surface group. If $\Gamma'$ is virtually a surface group then $\Gamma$ is a hyperbolic group with $\partial_{\infty}\Gamma=\partial_{\infty}\Gamma'\cong S^1$. By the work of Gabai \cite{Gabai}, it follows that $\Gamma$ is virtually a surface group.

Now suppose that $\Gamma'$ is virtually free. Note that $\Gamma$ is a finitely presented group, so by Dunwoody's accessibility theorem \cite{Dun}, $\Gamma$ splits as the fundamental group of a finite graph of groups $\mathcal{G}$ with edge groups of at most one end. Note that since $\partial_{\infty}\Gamma' \cong \partial_{\infty}\Gamma$ is totally disconnected, all vertex groups of $\mathcal{G}$ are virtually cyclic and thus $\Gamma$ itself is virtually free. 
\end{proof} 

We will use these in combination with the following topological characterization of the circle:

\begin{proposition}[cf.\ {\cite[p.\ 14]{Zimmer}}] \label{prop:not circle implies 2-eps connected}
If $M$ is a compact connected topological space that is not homeomorphic to $S^1$, then there exist distinct points $x, y \in M$ such that $M \smallsetminus \{x,y\}$ is connected.
\end{proposition}

\subsection{Remarks on torsion}

At various points in our arguments we will need to pass to finite-index torsion-free subgroups. Here we collect a few remarks that help iron out minor technicalities that may arise from this process.

\begin{proposition} \label{prop:nilp by finite is nilp}
Finite extensions of virtually nilpotent groups are virtually nilpotent. \end{proposition}

\begin{proof} The proposition follows immediately from Gromov's polynomial growth theorem \cite{Gromov-nil}, since finite extensions of virtually nilpotent groups still have polynomial growth. 

We also provide an elementary proof here for the interested reader. Suppose that $\Gamma/F$ is nilpotent, where $F$ is a finite normal subgroup of $\Gamma$. There exists a finite-index subgroup $\Gamma'$ of $\Gamma$ which is centralized by $F$: this is the kernel of the homomorphism $\Gamma \to \Aut(F)$ given by $g \mapsto (f \mapsto gfg^{-1})$. Since $\Gamma'/F$ is nilpotent, there exists $k>0$ such that $\mathfrak{g}_k(\Gamma')F/F=1$ or in other words $\mathfrak{g}_k(\Gamma') \subset F$, where $\mathfrak{g}_r(\Gamma')$, $r\in \mathbb{N}$, denotes the $r$\textsuperscript{th} term of the lower central series of $\Gamma'$. Since $[F,\Gamma']=1$, $\mathfrak{g}_{k+1}(\Gamma')=1$, so $\Gamma'$ is nilpotent.\end{proof}

The following proposition can be obtained by Bass--Serre theory or by elementary covering space theory. 

\begin{proposition}\label{graph-torsionfree} Let $\mathcal{G}$ be a finite graph of groups. Suppose that the vertex groups of $\mathcal{G}$ are virtually torsion-free and its edge groups are finite. Then $\pi_1(\mathcal{G})$ contains a torsion-free finite-index subgroup which is isomorphic to a free product of the form $\Gamma_1\ast \cdots \ast \Gamma_{m}$, where for each $i=1,\ldots,m$, $\Gamma_{i}$ is either infinite cyclic or isomorphic to a finite-index subgroup of some vertex group of $\mathcal{G}$. \end{proposition}

\subsection{Splittings and boundaries of relatively hyperbolic groups}

Let $\Gamma$ be a non-elementary relatively hyperbolic group with peripheral subgroups $\peripherals$.
Let $X(\Gamma,\peripherals)$ be a(ny) Gromov model associated to this relatively hyperbolic structure. Note that $\partial(\Gamma,\peripherals) \cong \partial_\infty X(\Gamma,\peripherals)$.

\begin{definition} We say that $\Gamma$ splits non-trivially over a group $H$ relative to $\peripherals$ if we can write $\Gamma$ as an amalgamated free product or HNN extension over $H$, where each subgroup in $\peripherals$ is conjugate into one of the factors.

Similarly, we say that $\Gamma$ splits (non-trivially) over finite groups relative to $\peripherals$ if $\Gamma$ splits as a (non-trivial) graph of groups, with finite edge groups and every peripheral subgroup conjugate into a vertex group. 
\end{definition}

\begin{theorem}[{\cite[Prop.\ 10.1-3]{Bowditch_relhyp}}, see also {\cite[\S3]{DahmaniGroves}}]
\label{thm:splittings from disconnected Bowditch boundary}
Let $(\Gamma,\peripherals)$ be a non-elementary relatively hyperbolic group. The boundary $\partial(\Gamma,\peripherals)$ is disconnected if and only if $(\Gamma,\peripherals)$ splits non-trivially over finite groups relative to $\peripherals$, such that no vertex group splits non-trivially over finite subgroups relative to its peripherals.

In such a splitting, each vertex group is hyperbolic relative to the peripheral subgroups that it contains.
\end{theorem}

Suppose we have a non-trivial splitting of $\Gamma$ over finite groups relative to $\peripherals$.
Then, from the discussion following Proposition 10.3 in \cite{Bowditch_relhyp}, there is a natural continuous inclusion of 
the boundary of each of the vertex subgroups (i.e.\ the limit set of each of these subgroups in $X(\Gamma,\peripherals) \cup \partial(\Gamma,\peripherals)$)
into $\partial(\Gamma,\peripherals)$.
Indeed, each component of $\partial(\Gamma,\peripherals)$ is either a single point or the boundary of an infinite non-peripheral vertex group.
Singleton components are either fixed
points of vertex groups which are peripheral, or else corresponds to 
ideal points of the tree corresponding to the splitting.

We remark that this is \emph{not} the JSJ splitting over elementary subgroups given by Guirardel--Levitt \cite{GuirardelLevittJSJ} and described in Haulmark--Hruska \cite{HaulmarkHruskaJSJ}: for instance, the edge groups in the JSJ splitting are not in general finite. The JSJ splitting can be viewed as a further canonical maximal splitting of each of the vertex groups with connected boundary.

\subsection{Relatively Anosov subgroups} Relatively Anosov representations $\rho \colon (\Gamma,\Pc) \to \SL_d(\Rb)$ may be defined in terms of limit maps, which are $\rho$-equivariant homeomorphisms between the Bowditch boundary $\partial(\Gamma,\Pc)$ and a closed $\rho(\Gamma)$-invariant subset of the flag variety, which we will call the flag limit set. 

For $1\leq k\leq d-1$, denote by $\mathsf{Gr}_{k}(\mathbb{R}^d)$ the Grassmannian of $k$-planes in $\Rb^d$.

\begin{definition}[\cite{ZZ1}]\label{definition-relAnosov}
{\em For $d\in \mathbb{N}$ and $1\leq k\leq \frac{d}{2}$, a representation $\rho\colon(\Gamma, \mathcal{P})\to \mathsf{SL}_d(\Rb)$ is \emph{relatively $k$-Anosov} if admits a pair of continuous $\rho$-equivariant maps $$(\xi_{\rho}^k,\xi_{\rho}^{d-k}):\partial(\Gamma, \mathcal{P})\rightarrow \mathsf{Gr}_k(\mathbb{R}^d)\times \mathsf{Gr}_{d-k}(\mathbb{R}^d)$$  called {\em the limit maps of $\rho$}, with the following properties:
\medskip

\noindent \textup{(i)} $\xi_{\rho}^k$ and $\xi_{\rho}^{d-k}$ are \emph{transverse}: for every $x,y\in \partial(\Gamma,\mathcal{P})$, $x,\neq y$, we have $$\mathbb{R}^d=\xi_{\rho}^k(x)\oplus \xi_{\rho}^{d-k}(y).$$

\noindent \textup{(ii)} $\xi_\rho^k$ and $\xi_\rho^{d-k}$ are \emph{strongly dynamics-preserving}: given $(\gamma_n)_{n\in \mathbb{N}}\subset \Gamma$ such that $\gamma_n \to x\in \partial(\Gamma, \mathcal{P})$ and $\gamma_n^{-1} \to y \in \partial(\Gamma,\mathcal{P})$, $\rho(\gamma_n) V \to \xi^k(x)$ for all $V\in \mathsf{Gr}_{k}(\mathbb{R}^d)$ such that $\mathbb{R}^d=V\oplus \xi^{d-k}(y)$.}
\end{definition} 

Additional properties of the limit maps follow from the ones in the definition:
\medskip

\noindent \textup{(iii)} $\xi_{\rho}^k$ and $\xi_{\rho}^{d-k}$ are {\em compatible}, i.e. $\xi_{\rho}^{k}(x)\subset \xi_{\rho}^{d-k}(x)$ for every $x\in \partial(\Gamma,\mathcal{P})$.\\
\noindent \textup{(iv)} $\xi_{\rho}^{k}$ (resp. $\xi_{\rho}^{d-k}$) are {\em dynamics-preserving} at proximal points, i.e.\ all loxodromic $\gamma \in (\Gamma,\peripherals)$ are sent to biproximal elements $\rho(\gamma) \in \SL_d(\Rb)$, and $\xi_\rho^k(\gamma^+) \in \mathsf{Gr}_k(\mathbb{R}^d)$ (resp. $\xi_\rho^{d-k}(\gamma^+) \in \mathsf{Gr}_{d-k}(\mathbb{R}^d)$) is the unique attracting fixed point of $\rho(\gamma)$ in $\mathsf{Gr}_{i}(\mathbb{R}^d)$ (resp. $\mathsf{Gr}_{d-k}(\mathbb{R}^d)$).
\medskip

Below we refer to limit sets in the Bowditch boundary, and by the compatibility and transversality of the limit maps, this is equivalent to talking about subsets of the flag limit set. 

\begin{proposition} \label{prop:peripheral unique flags} Let $(\Gamma, \mathcal{P})$ be a non-elementary relatively hyperbolic group which is virtually torsion-free. The limit set of a subgroup $\Gamma'$ of $\Gamma$ is a singleton if and only if $\Gamma'$ is virtually contained in a conjugate of a peripheral subgroup.
\begin{proof}
A peripheral subgroup has a singleton limit set in the Bowditch boundary. Conversely, if a subgroup has a singleton limit set, then it cannot contain any infinite-order non-peripheral elements.
\end{proof}
\end{proposition}

\begin{proposition}[{\cite[\S8]{ZZ1}}] \label{prop:relA peripherals are v nilpotent}
Given $\rho\colon (\Gamma,\peripherals) \to \SL_d(\Rb)$ a relatively Anosov representation, and $P \in \peripherals$, then $P$ is virtually nilpotent.\end{proposition}

Similar to the case of Anosov representations of hyperbolic groups, images of relatively Anosov representations also contain Anosov free groups as subgroups.

\begin{lemma}\label{free} Let $(\Gamma,\peripherals)$ be a finitely generated relatively hyperbolic group, and suppose that $\rho\colon (\Gamma,\peripherals) \to \GL_d(\Rb)$ is a relatively $k$-Anosov representation and $\Gamma_0$ is a non-elementary subgroup of $\Gamma$. Then there exist $a,b\in \Gamma_0$ such that $\big \langle \rho(a),\rho(b) \big \rangle$ is a $k$-Anosov, rank 2 free subgroup of $\SL_d(\Rb)$.

\begin{proof} By passing to exterior powers we may assume that $k=1$, see \cite[Def.\ 10.2]{reldomreps} and \cite[Ex.\ 13.2]{ZZ1}). Since $\Gamma_0$ is non-elementary there exists $\gamma \in \Gamma$ such that $\rho(\gamma)$ is $1$-biproximal (see \cite[Prop.\ 4.6]{reldomreps} and \cite[Prop.\ 4.2]{ZZ1}), and $\xi_{\rho}^1(\gamma^{+})$ and $\xi_{\rho}^{d-1}(\gamma^{+})$ (resp. $\xi_{\rho}^1(\gamma^{-})$ and $\xi_{\rho}^{d-1}(\gamma^{-})$) are the attracting fixed points of $\rho(\gamma)$ (resp. $\rho(\gamma^{-1})$) in $\proj(\Rb^d)$ and $\mathsf{Gr}_{d-1}(\Rb^d)$ respectively. Now since $\Lambda_{\Gamma_0}$ is infinite we may find $\delta \in \Gamma_0$ such that $\{\delta \gamma^{+},\delta \gamma^{-}\} \cap \{\gamma^{+},\gamma^{-}\}$ is empty. Since, up to subsequence, $\lim_{n} \gamma^{\pm n}\delta^{\pm 1}\gamma^{+}=\lim_n \gamma^{\pm n}\delta^{\pm 1}\gamma^{-}=\gamma^{\pm}$, we have $$\rho(\delta^{\pm 1})\xi_{\rho}^1(\gamma^{+}), \rho(\delta^{\pm 1})\xi_{\rho}^1(\gamma^{-}) \in \mathbb{P}(\mathbb{R}^d)\smallsetminus \left( \proj(\xi_{\rho}^{d-1}(\gamma^{+}))\cup \proj(\xi_{\rho}^{d-1}(\gamma^{-})) \right).$$
This shows that the set $\big \{\rho(\gamma),\rho(\gamma^{-1}), \rho(\delta \gamma \delta^{-1}),\rho(\delta \gamma^{-1}\delta^{-1})\big \}$ is well positioned in the sense of \cite[App.\ A]{CLS}. Then by a result of Kapovich--Leeb--Porti \cite[Thm.\ 7.40]{KLP} and Canary--Lee--Stover--Sambarino \cite[Thm.\ A2]{CLS}, there exists $N>1$ such that the group $\big \langle \rho(\gamma^N), \rho(\delta \gamma^N\delta^{-1})\big\rangle$ is a free and $1$-Anosov subgroup of $\SL_d(\Rb)$. \end{proof}
\end{lemma}

\subsection{Relatively Anosov subgroups preserving domains}

The following result was established for hyperbolic groups with 1-Anosov representations in \cite[Prop.\ 5.3]{IZ} (see also \cite[Thm.\ 3.1]{Zimmer} and \cite[Prop.\ 2.8]{CanaryTsouvalas}), and also holds in the relative case with essentially the same proof.
\begin{proposition}\label{prop:rel Anosov preserving domain}
% [{cf.\ \cite[Prop.\ 5.3]{IZ}}]
Let $(\Gamma,\peripherals)$ is a relatively hyperbolic group such that $\partial(\Gamma,\peripherals)$ is connected. Suppose $\rho\colon (\Gamma,\peripherals) \to
\PGL_d(\Rb)$ is a relatively 1-Anosov representation, and 
$$ \left( \xi_{\rho}^1, (\xi_\rho^1)^* \right) \colon \partial(\Gamma,\peripherals) \to \proj(\Rb^d) \times \proj((\Rb^d)^*)$$
are the associated limit maps.

\begin{enumerate}
\item If $\partial(\Gamma,\peripherals)$ is not homeomorphic to $S^1$, then $\xi_{\rho}^1(\partial(\Gamma,\peripherals))$ is bounded in some affine chart of $\proj(\Rb^d)$.

\item If $\xi_{\rho}^1(\partial(\Gamma,\peripherals))$ is bounded in some affine chart of $\proj(\Rb^d)$, then $\rho(\Gamma)$ preserves a properly convex domain in $\proj(\mathrm{span}(\xi_{\rho}^1(\partial(\Gamma,\peripherals)))$.

\item If $\xi_{\rho}^1(\partial(\Gamma,\peripherals))$ is bounded in some affine chart of $\proj(\Rb^d)$, then $\rho(\Gamma)$ preserves a properly convex domain in $\proj(\Rb^d)$.
\end{enumerate}
\end{proposition}

\begin{proof}
The proof of (1) follows that of \cite[Prop.\ 5.3(1)]{IZ}.

The proof of (2) follows the arguments in \cite{Zimmer} or \cite[Prop.\ 2.8]{CanaryTsouvalas}.

For the proof of (3), we modify the proof of \cite[Prop.\ 5.3(2)]{IZ}, noting that the result of \href{https://www.ams.org/journals/era/1996-02-02/S1079-6762-96-00013-3/S1079-6762-96-00013-3.pdf}{Swarup} used in that argument --- if a hyperbolic group has connected Gromov boundary $\partial_\infty\Gamma$, then $\partial_\infty\Gamma$ has no (global) cut-points --- is no longer true in the relative case.
Instead, we use Proposition \ref{prop:not circle implies 2-eps connected}, which gives a non-cut pair $\{x,y\}$. 

Fix an affine chart $\mathbb{A}$ containing $\xi_{\rho}^1(\partial(\Gamma,\peripherals))$. 
Working in $\proj(\mathrm{span}(\xi_{\rho}^1(\partial(\Gamma,\peripherals)))$, and arguing as in \cite[\S3.1]{Zimmer}, we obtain a $\Gamma$-invariant properly convex domain $\Omega_0 \subset \proj(\mathrm{span}(\xi_{\rho}^1(\partial(\Gamma,\peripherals)))$ that $\Omega_0$ avoids $\ker ((\xi_{\rho}^1)^{*}(x))$ and $\ker((\xi_{\rho}^1)^{*}(y))$. Since $\Omega_0$ is $\rho(\Gamma)$-invariant, it also avoids all of the $\rho(\Gamma)$-translates of $\ker((\xi_{\rho}^1)^*(x))$ and $\ker((\xi_{\rho}^1)^*(y))$. By minimality of the $\Gamma$-action on $\partial(\Gamma,\peripherals)$ and continuity and equivariance of the limit maps, we conclude that $\Omega_0 \cap \ker \xi_\rho^*(z) = \varnothing$ for all $z \in \partial(\Gamma,\peripherals)$. 

Then take a bounded neighborhood $N \supset \Omega_0$ in our affine chart $\mathbb{A}$. Arguing as in \cite[Lem.\ 5.13]{IZ}, given any $p \in \Omega_0$ we may find a neighborhood $U \subset \mathbb{A}$ of $p$ whose $\rho(\Gamma)$-orbit is contained in $N$. Then we may take our convex domain to be the (relative interior of the) convex hull of $\xi_{\rho}^1(\partial(\Gamma,\peripherals)) \cup \rho(\Gamma) \cdot N$.
\end{proof}

\section{Proofs of Theorem \ref{thm:SL3R'} and Theorem \ref{symplectic}}

Fix $(\Gamma,\peripherals)$ a finitely-generated non-elementary relatively hyperbolic group.

\subsection{Convergence groups and limit maps}

We start with the short observation that representations of convergence groups admitting non-constant limit maps are necessarily discrete and faithful.

\begin{observation}\label{obs1} Let $X$ be a compact metrizable space and $\Gamma$ be a finitely-generated group acting (minimally) as a non-elementary convergence group on $X$. Suppose $\rho\colon \Gamma \to \GL_d(\Rb)$ is a representation which admits a non-constant continuous $\rho$-equivariant map $\xi\colon X\to \proj(\Rb^d)$. 

Then the representation $\hat{\rho}\colon \Gamma \to \SL_d^{\pm}(\Rb)$ defined by $\hat{\rho}(\gamma) = \left|\det(\rho(\gamma))\right|^{-\frac{1}{d}} \rho(\gamma)$ has finite kernel and discrete image. In particular, $\rho$ has discrete image and finite kernel.

\begin{proof} Note that the map $\xi$ is also $\hat{\rho}$-equivariant. If $\hat{\rho}$ is not discrete or has infinite kernel, then there exists an infinite sequence of elements $(\gamma_n)_{n\in \mathbb{N}}$ of $\Gamma$ such that $\lim_n  \hat{\rho}(\gamma_n)=\textup{I}_d$. Note that up to passing to a subsequence there exist $\eta_1,\eta_2\in X$ such that $\lim \gamma_n x=\eta_1$ for every $x\in X\smallsetminus \{\eta_2\}$. Then observe that $\xi(x)=\lim_n \rho(\gamma_n)\xi(x)=\lim_n \xi(\gamma _n x)=\xi(\eta_1)$ for every $x\in X\smallsetminus \{\eta_2\}$. In particular, by continuity, and because $X$ is perfect and hence $\eta_2$ is not isolated, $\xi$ is constant and this is a contradiction. It follows that $\hat{\rho}(\Gamma)$ is a discrete subgroup of $\GL_d(\Rb)$ and $\hat{\rho}$ has finite kernel.
\end{proof} \end{observation}

\subsection{The \texorpdfstring{$\SL_3(\Rb)$}{SL(3,R)} case} For the proof of Theorem \ref{thm:SL3R'} we prove first the following theorem for relatively Anosov representations into $\mathsf{SL}_3(\mathbb{R})$.

\begin{theorem} \label{thm:SL3R} Suppose that $\rho\colon (\Gamma,\mathcal{P}) \to \SL_3(\Rb)$ is a relatively $1$-Anosov representation. Let $\Gamma_0$ be a finitely generated subgroup of $\Gamma$ whose limit set $\Lambda_{\Gamma_0}$ in the Bowditch boundary $\partial(\Gamma, \peripherals)$ is connected. Then one of the following holds:\\
\noindent \textup{(i)} $\Gamma_0$ is a virtually nilpotent group and $\Lambda_{\Gamma_0}$ is a singleton, or\\
\noindent \textup{(ii)} $\Gamma_0$ is virtually a free group, or\\
\noindent \textup{(iii)} $\partial(\Gamma,\peripherals)=\Lambda_{\Gamma_0}\cong S^1$, $\Gamma$ is virtually a surface group and $\Gamma_0$ has finite index in $\Gamma$.

\begin{proof} 
{\bf We first deal with the case where $\rho$ is a faithful representation.}

If $\xi_{\rho}^1\colon\partial(\Gamma,\mathcal{P})\rightarrow \mathbb{P}(\mathbb{R}^3)$ restricted to $\Lambda_{\Gamma_0}\subset \partial(\Gamma,\peripherals)$ is constant, then $\Gamma_0$ has to be a virtually nilpotent group by Proposition \ref{prop:peripheral unique flags} and Proposition \ref{prop:relA peripherals are v nilpotent}. Suppose this is not the case, so that $\dim V_0 \geq 2$ where $V_0:=\textup{span}(\xi_{\rho}^1(\Lambda_{\Gamma_0}))$.
Let us assume that $\Gamma_0$ is not virtually a free group; we will prove that (iii) holds. 

{\bf We first claim that $\Lambda_{\Gamma_0}=\partial(\Gamma,\peripherals)$.} Suppose that this does not happen and that there is $\eta\in \partial(\Gamma,\peripherals)\smallsetminus \Lambda_{\Gamma_0}$. Then, since $\xi_{\rho}^1$ and $\xi_{\rho}^2$ are transverse, the compact connected set $\xi^1_{\rho}(\Lambda_{\Gamma_0})$ is contained in the affine chart $\mathbb{A}_{\eta}:=\proj(\Rb^3)\smallsetminus \proj(\xi_{\rho}^2(\eta))$. 
By Observation \ref{obs1} the restriction $\rho_{V_0}\colon \Gamma_0 \to \GL(V_0)$ has discrete image and finite kernel. There are two cases to consider: either $\textup{dim}(V_0)=2$, or $\textup{dim}(V_0)=3$.

Suppose $\dim V_0 = 2$. Then $\rho_{V_0}(\Gamma_0)$ is Fuchsian. Since $\Gamma_0$ is assumed to be not virtually free, necessarily $\xi_{\rho}^1(\Lambda_{\Gamma_0})=\proj(V_0)$. However, this is impossible since $\xi_{\rho}^1(\Lambda_{\Gamma_0})$ lies in the affine chart $\mathbb{A}_{\eta}$ of $\proj(\Rb^3)$. 

Now suppose $\dim V_0 = 3$. The convex hull $\mathcal{C}_0$ of the connected compact set $\xi^1(\Lambda_{\Gamma_0})$ in $\mathbb{A}_{\eta}$ has non-empty interior. Since $\Lambda_{\Gamma_0}$ is connected and preserved by $\rho(\Gamma_0)$, it follows that $\rho(\Gamma_0)$ acts properly discontinuously on $\Omega_0:=\textup{int}(\mathcal{C}_0)$. In particular, since $\rho(\Gamma_0)$ is not virtually free, it is virtually a surface group, $\rho(\Gamma_0)$ acts cocompactly on $\Omega_0$ and $\rho|_{\Gamma_0}$ is $1$-Anosov. In fact, $\Gamma_0$ acts minimally on $\Lambda_{\Gamma_0}$ and $\xi_{\rho}^{1}|_{\Lambda_{\Gamma_0}}$ and $\xi_{\rho}^{2}|_{\Lambda_{\Gamma_0}}$ are the Anosov limit maps of $\rho|_{\Gamma_0}$. Moreover, note that $\rho(\Gamma_0)$ is irreducible and $$\Omega_{0}=\proj(\Rb^3) \smallsetminus \bigcup_{x\in \Lambda_{\Gamma_0}}\proj(\xi_{\rho}^2(x))
$$ is contained in $\mathbb{A}_{\eta}$. In other words, we have that $$\proj(\xi_{\rho}^2(\eta))\subset \bigcup_{x\in \Lambda_{\Gamma_0}}\proj(\xi_{\rho}^2(x)) 
$$ which violates transversality, since $\xi_{\rho}^1(\eta)\oplus \xi_{\rho}^2(x)=\Rb^3$ for every $x\in \Lambda_{\Gamma_0}$. Therefore, we have $\Lambda_{\Gamma_0}=\partial(\Gamma,\peripherals)$. 

{\bf We now claim that $\partial(\Gamma,\peripherals)\cong S^1$.} Suppose not. By Propositions~\ref{prop:not circle implies 2-eps connected} and~\ref{prop:rel Anosov preserving domain}(1), we may deduce that $\xi_{\rho}^1(\partial(\Gamma,\peripherals))$ lies in an affine chart $\mathbb{A}\subset \proj(\Rb^3)$. 
Writing $V_0:=\textup{span}(\xi_{\rho}^1(\partial(\Gamma,\peripherals)))=\textup{span}(\xi_{\rho}^1(\Lambda_{\Gamma_0}))$ as before, we cannot have $\dim V_0 = 2$ by the same argument as in the previous step.
Hence $\dim V_0 =3$ and $\rho(\Gamma)$ preserves the properly convex open set $\textup{int}(\mathcal{C}_{\rho})$, where $\mathcal{C}_{\rho}$ is the convex hull of $\xi_{\rho}^1(\partial(\Gamma,\peripherals))$ in $\mathbb{A}$. 
Arguing as in the previous step, $\rho(\Gamma)$ is virtually a surface group 
% and $\rho(\Gamma)$ acts cocompactly on $\Omega_0$ 
and $\rho$ is $1$-Anosov. In particular, $\partial(\Gamma,\peripherals)$ identifies with the limit set of $\rho(\Gamma)$ which is a circle. This is a contradiction. 

{\bf It finally follows that if neither (i) or (ii) holds then $\partial(\Gamma, \peripherals) \cong \Lambda_{\Gamma_0} \cong S^1$.} In that case, by Theorem~\ref{thm:convergence groups on S1 are Fuchsian}, both $\Gamma$ and $\Gamma_0$ are virtually surface groups, and $\Gamma_0$ is necessarily a finite-index subgroup of $\Gamma$ (otherwise $\Lambda_{\Gamma_0}$ will be smaller).

{\bf It remains to deal with the case where $\rho$ is non-faithful.} In this case, note that $\ker\rho$ is finite. Then the above arguments all apply for the quotients $\Gamma':=\Gamma / \ker \rho$ and $\Gamma'_0 := \Gamma_0 / \ker \rho$, and show that one of the following holds: either $\Gamma'_0$ is virtually nilpotent or virtually free, or $\Gamma'$ is virtually a surface group. The conclusion then follows by Proposition~\ref{prop:nilp by finite is nilp} and Theorem~\ref{GabDun}. \end{proof}
\end{theorem}

This concludes the proof of Theorem \ref{thm:SL3R'} for the case where $\partial(\Gamma,\mathcal{P})$ is connected.

\begin{remark} We repeatedly use the following fact that a subgroup $\mathsf{\Gamma}<\GL_d(\Rb)$ which preserves a connected compact subset $\mathcal{C}$ of an affine chart $\mathbb{A}$ with $\mathbb{R}^d=\textup{span}(\mathcal{C})$, then $\mathsf{\Gamma}$ preserves a properly convex open domain in $\proj(\Rb^d)$, namely the interior of the convex hull of $\mathcal{C}$ in $\mathbb{A}$. 

This fact fails to be true if we drop the assumption that $\mathcal{C}$ is connected. For example, let $\Sigma_r$ be the closed oriented surface of genus $r \geq 2$. If $j\colon \pi_1(\Sigma_r)\xhookrightarrow{} \SL_2(\Rb)$ is a discrete and faithful representation and $H<\pi_1(\Sigma_r)$ is a free subgroup of rank at least two, then the product $(j \times 1)|_{H}$, where $1$ denotes the trivial representation into $\SL_1(\Rb)$, admits Zariski-dense Anosov deformations $j_{\varepsilon}\colon H\to \SL_3(\Rb)$ whose limit set are contained in an affine chart of $\proj(\Rb^3)$. However, $j_{\varepsilon}(H)$ cannot preserve a properly convex domain in $\proj(\Rb^3)$, since $j_{\varepsilon}(H)$ contains proximal elements whose eigenvalue of maximum modulus is negative.
\end{remark}

 \subsection{Anosov representations into symplectic groups}

For a $1$-proximal matrix $g\in \mathsf{GL}_d(\mathbb{R})$ we denote by $\lambda_1(g)\in \mathbb{R}$ its unique eigenvalue of maximum modulus. A proximal matrix $g$ is called {\em postitively $1$-proximal} if $\lambda_1(g)>0$.

For the proof of Theorem \ref{symplectic} we first prove the following result for relatively Anosov representations into $\Sp_{2m}(\mathbb{R})$.

\begin{theorem} \label{2q+1} Suppose that $\rho\colon (\Gamma,\mathcal{P}) \to \Sp_{2m}(\Rb)$ is a relatively $k$-Anosov representation for some $1\leq k \leq m$ odd integer. Let $\Gamma_0$ be a subgroup of $\Gamma$ whose limit set $\Lambda_{\Gamma_0}$ in the Bowditch boundary $\partial(\Gamma, \peripherals)$ is connected. Then one of the following holds:\\
\noindent \textup{(i)} $\Gamma_0$ is a virtually nilpotent group and $\Lambda_{\Gamma_0}$ is a singleton, or\\
\noindent \textup{(ii)} $\Gamma_0$ is virtually free group, or \\ 
\noindent \textup{(iii)} $\partial(\Gamma,\peripherals)=\Lambda_{\Gamma_0}\cong S^1$, $\Gamma$ is virtually a surface group and $\Gamma_0$ has finite index in $\Gamma$.
\end{theorem}

\begin{proof}[Proof of Theorem \ref{2q+1}] The strategy of proof is similar to the $\SL_3(\Rb)$ case, although the details in the last step differ. We assume that $\rho$ is faithful, otherwise we use the same reduction as at the end of the proof for the $\SL_3(\Rb)$ case. Since the $k$-exterior power $\wedge^k \colon \Sp_{2m}(\mathbb{R})\rightarrow \mathsf{GL}(\wedge^{k}\mathbb{R}^{2m})$ preserves a symplectic form on $\wedge^{k}\mathbb{R}^{2m}$ and $\wedge^k \rho$ is relatively $1$-Anosov, we may assume that $k=1$. Let $\xi_{\rho}^1,\xi_{\rho}^{2m-1}$ be the limit maps of $\rho$ and suppose $\xi_{\rho}^{1}(\Lambda_{\Gamma_0})$ is not a singleton and $\Gamma_0$ is not virtually a free group; we will prove that (ii) holds.

First, by Lemma \ref{free}, we may find $a,b \in \Gamma$ such that $\rho|_{\langle a,b\rangle}\colon\langle  a,b\rangle \rightarrow \Sp_{2m}(\mathbb{R})$ is $1$-Anosov. Let us prove that $\Lambda_{\Gamma_0}=\partial(\Gamma, \peripherals)$. If not, fix $\eta\in \partial(\Gamma, \peripherals)\smallsetminus \Lambda_{\Gamma_0}$ and by transversality we have that $$\xi^{1}_{\rho}(\Lambda_{\Gamma_0})\subset \proj(\mathbb{R}^{2m})\smallsetminus \proj (\xi_{\rho}^{2m-1}(\eta)).$$ Let $V_1:=\textup{span}(\xi_{\rho}^{1}(\Lambda_{\Gamma_0}))$ and observe that $\xi_{\rho}^{1}(\Lambda_{\Gamma_0}) $ is a connected compact subset of the affine chart $B_{\eta}:=\mathbb{P}(V_1)\smallsetminus \mathbb{P}(V_1)\cap \mathbb{P}(\xi_{\rho}^{2m-1}(\eta))$ of $\proj(V_1)$. In particular, $\rho|_{V_1}(\Gamma_0)$ (and hence $\rho|_{V_1}(\langle a,b\rangle)$) preserves a properly convex domain $\Omega_{0}\subset \proj(V_1)$. We deduce that there exists a finite-index subgroup $H_0 \subset \langle a,b\rangle$ such that $\rho|_{V_1}(H_0)$ preserves a properly convex cone $C_1\subset V_1$ and hence $\rho|_{V_1}(H_0)$ is positively proximal (see \cite{Ben}). 

Since every $h\in H_0\smallsetminus \{1\}$ admits an eigenvector of maximum modulus in $V_1$ (because the limit set of $\rho(H_0)$ is contained in $\proj(V_1)$), we deduce that $\lambda_1 (\rho(h))>0$. Since $\rho|_{H_0}\colon H_0\to \Sp_{2m}(\mathbb{R})$ is $1$-Anosov, by the openness of Anosov representations \cite{GW, Lab}, we may find a continuous path of representations $\left\{\rho_t\colon H_0\to \Sp_{2m}(\Rb) \right\}_{t\in [0,1]}$ such that $\rho_0=\rho|_{H_0}$, $\rho_1$ is Zariski-dense in $\Sp_{2m}(\mathbb{R})$ and $\rho_t$ is $1$-Anosov for every $0\leq t \leq 1$. Now for $h\in H_0\smallsetminus\{1\}$, $\rho_t(h)$ is $1$-proximal, so by continuity we obtain $$\lambda_1 (\rho_0(h))\lambda_1(\rho_1(h))=\lambda_1 (\rho(h))\lambda_1(\rho_1(h))>0.$$ 

We deduce that $\rho_1(H_0)$ is a strongly irreducible subgroup of $\GL_{2m}(\Rb)$ (since it is a Zariski dense subgroup of $\Sp_{2m}(\mathbb{R})$) all of whose non-trivial elements are positively $1$-proximal. Thus, by \cite{Ben}, $\rho_1(H_0)$ preserves a properly convex cone $\mathcal{C}'\subset \Rb^{2m}$. However, this is contradiction by Benoist's work, see \cite[Cor. 3.5]{Ben}. 

Therefore, we conclude that $\partial(\Gamma, \peripherals)=\Lambda_{\Gamma_0}$. In fact, our arguments show that $\xi_{\rho}^{1}(\Lambda_{L}))$ is not bounded in any affine chart of $\proj(\Rb^{2m})$, as soon as $\Lambda_{L}$ is connected and $L$ is not contained in a conjugate of a peripheral subgroup of $\Gamma$.

 It remains to show that $\Lambda_{\Gamma_0}\cong S^1.$ If not, we may use again Propositions~\ref{prop:not circle implies 2-eps connected} and~\ref{prop:rel Anosov preserving domain}(1)
 % , together with the transversality of the limit maps 
 to deduce that $\xi_{\rho}^{1}(\partial(\Gamma,\peripherals))$ is bounded in an affine chart of $\proj(\Rb^{2m})$. However, this is impossible by the previous arguments. Hence $\partial(\Gamma, \peripherals)\cong S^1$ and (iii) holds. \end{proof}

This completes the proof of Theorem \ref{symplectic} for the case where $\partial(\Gamma,\mathcal{P})$ is connected.

\subsection{The case of the disconnected boundary} \label{sec:disconnected boundary} Now we give the proof of Theorem \ref{thm:SL3R'} and Theorem \ref{(2q+1)'} in the case where $\partial(\Gamma, \mathcal{P})$ is disconnected.

\begin{proof} Fix a finitely-generated relatively hyperbolic group $(\Gamma,\peripherals)$ and assume that $\partial(\Gamma,\peripherals)$ is disconnected. We assume now that there is a representation $\rho \colon (\Gamma,\peripherals) \to \SL_3(\Rb)$ (resp. $\rho \colon (\Gamma,\peripherals) \to \Sp_{2m}(\mathbb{R})$) which is relatively $1$-Anosov (resp. relatively $k$-Anosov for some odd $1\leq k \leq m$).

By Theorem~\ref{thm:splittings from disconnected Bowditch boundary}, there exists a maximal splitting $\Gamma=\pi_1(\mathcal{G})$ of $\Gamma$ relative to $\mathcal{P}$, where $\mathcal{G}$ is a finite graph of groups whose edge groups are finite and whose vertex groups are relatively hyperbolic with connected Bowditch boundary and connected limit set in $\partial(\Gamma, \mathcal{P})$. In this case, Theorem \ref{thm:SL3R} (resp.\ Theorem \ref{2q+1}) implies that all of the vertex subgroups of $\mathcal{G}$ either (i) have singleton boundary, or (ii) are virtually free, or (iii) are virtually surface groups. 

In case (iii), if there is a vertex group which is virtually a surface group, it follows by Theorem \ref{thm:SL3R} (resp. Theorem \ref{2q+1}) that $\partial(\Gamma, \mathcal{P})$ is a circle and $\Gamma$ is virtually a surface group, which cannot happen since we assumed that $\partial(\Gamma, \mathcal{P})$ is disconnected. 

Therefore, each of the vertex groups of $\mathcal{G}$ are either virtually free or have singleton boundary. In case (i), the vertex group must be virtually peripheral and hence virtually nilpotent by Proposition \ref{prop:relA peripherals are v nilpotent}. Therefore, by Proposition \ref{graph-torsionfree}, we conclude that $\Gamma=\pi_1(\mathcal{G})$ is virtually a free product of cyclic groups with finite-index subgroups of the vertex groups of $\mathcal{G}$. In particular, $\Gamma$ is virtually a free product of nilpotent groups.

In the particular case where of a relatively $1$-Anosov representation $\rho\colon (\Gamma, \mathcal{P})\rightarrow \mathsf{SL}_3(\mathbb{R})$, by \cite[Cor.\ 1.5]{DoubaT}, $\Gamma$ is Gromov hyperbolic, every nilpotent vertex group of $\mathcal{G}$ is virtually cyclic, and hence $\Gamma$ is virtually free. Therefore, if $\rho\colon (\Gamma, \peripherals)\to \mathsf{SL}_3(\mathbb{R})$ is relatively Anosov and $\partial(\Gamma,\peripherals)$ is disconnected, then $\Gamma$ is virtually free.\end{proof}

\section{Proof of Theorem~\ref{thm:SL4R}} \label{sec:SL4R low cd}

Again fix $(\Gamma,\peripherals)$ a finitely-generated non-elementary relatively hyperbolic group. 

We start with the case where $\partial(\Gamma,\Pc)$ is connected.

\begin{theorem} \label{thm:SL4R connected}
Suppose $\rho\colon (\Gamma,\peripherals) \to \SL_4(\Rb)$ is a relatively 1-Anosov representation. Let $\Gamma_0$ be a finitely generated relatively hyperbolic subgroup of $\Gamma$ whose limit set $\Lambda_{\Gamma_0}$ in the Bowditch boundary $\partial(\Gamma,\Pc)$ is connected. Then one of the following holds:\\
\noindent \textup{(i)} $\Gamma_0$ is a virtually nilpotent group and $\Lambda_{\Gamma_0}$ is a singleton, or\\
\noindent \textup{(ii)} $\Gamma_0$ is virtually a free group or surface group, or\\
\noindent \textup{(iii)} $\Gamma_0$ is a finite extension of a group which is hyperbolic relative to \textup{(}a possibly empty collection of\textup{)} virtually free abelian subgroups of rank at most 2 and virtually the fundamental group of a compact 3-manifold \textup{(}possibly with boundary\textup{)}.
\end{theorem}

As in the proof of Theorem~\ref{thm:SL3R}, the bulk of the argument involves looking at what happens in the case where $\Lambda_{\Gamma_0}$ is neither a singleton nor a circle, and hence we are not in cases (i) or (ii). To this end, we will establish a few lemmas about subgroups of $\Gamma$ with connected limit set in $\partial(\Gamma,\Pc)$ not a point and not homeomorphic to a circle.
We first prove the following result, which may be seen as a relatively easier case of \cite[Thm.\ 8.1]{HruskaWalsh}. We note that all of our groups and peripheral subgroups are finitely generated.
\begin{lemma} \label{lem:planar + convergence action = surface peripherals}
Suppose $(\Gamma,\peripherals)$ is a non-elementary relatively hyperbolic group, and $\xi\colon \partial(\Gamma,\peripherals) \to \proj(\Rb^d)$ is a continuous equivariant embedding with image contained in a subset $S \subset \proj(\Rb^d)$ homeomorphic to a 2-sphere, such that $\Gamma$ acts on $S$ as a geometrically finite convergence group with maximal parabolic subgroups $\peripherals$ and limit set $\xi(\partial(\Gamma,\peripherals))$. 

Then all of the peripheral subgroups $P \in \peripherals$ are virtually surface groups or free groups.
\end{lemma}
\begin{proof}
By our hypotheses, the parabolic action of each $P \in \peripherals$ on $\xi(\partial(\Gamma,\peripherals))$ extends to a parabolic action of $P$ on all of the 2-sphere $S$.
Hence $P$ acts properly on the plane $S \smallsetminus \{x_P\}$, where $x_P$ denotes the fixed point of $P$, and therefore $P$ is virtually a surface group or free group (see \cite[Cor.\ 3.2]{HruskaWalsh}).
\end{proof}

This will be used in conjunction with the following
\begin{proposition}[\cite{rAGF_Mitul}] \label{prop:relAnosov implies DGF}
Suppose $\rho\colon (\Gamma,\peripherals) \to \PGL_d(\Rb)$ is a relatively 1-Anosov representation with limit map $\xi_{\rho}^1$, such that $\rho(\Gamma)$ preserves a properly convex domain $\Omega \subset \proj(\Rb^d)$. 

Then there exists a domain $\Omega_1 \subset \Omega$ such that $\rho(\Gamma)$ acts on $\partial\Omega$ as a geometrically finite convergence group with limit set $\xi_{\rho}^1(\partial(\Gamma,\peripherals))$.
\end{proposition}

\begin{proposition} \label{prop:connected bowditch boundaries for 1-Anosov in SL4R}
Suppose $(\Gamma,\peripherals)$ is a non-elementary relatively hyperbolic group, that $\rho \colon (\Gamma,\peripherals) \to \SL_4(\Rb)$ is 1-Anosov relative to $\peripherals$, and that $\partial(\Gamma,\peripherals)$ is connected and not homeomorphic to a circle.

Then each of the peripherals $P \in \peripherals$ is free abelian of rank at most 2.
\end{proposition}

\begin{proof} 
By Proposition \ref{prop:rel Anosov preserving domain}, $\rho(\Gamma)$ preserves a properly convex domain. In particular. $\xi(\partial(\Gamma,\peripherals))$ embeds in the boundary of this domain and hence $\partial(\Gamma,\peripherals)$ is planar. Indeed, from Proposition \ref{prop:relAnosov implies DGF}, $\rho(\Gamma)$ acts on the 2-sphere as a geometrically finite convergence group.
 
Then, by Lemma \ref{lem:planar + convergence action = surface peripherals}, each of the peripherals is virtually a free group or virtually a surface group.
Since the peripherals in $\peripherals$ must be virtually nilpotent for $\rho$ to be 1-Anosov relative to $\peripherals$, we conclude that each of the peripherals is virtually free abelian, of rank at most 2.
\end{proof}
% \end{corollary}

We might hope to conclude that $\Gamma$ is a Kleinian group. This could follow from an extension to the Cannon conjecture, which remains a major open question. (See \cite[Conj.\ 1.1]{HruskaWalsh} for a statement given the additional assumption that $\partial(\Gamma,\peripherals)$ has no cut points.) In our case, we can at least obtain that $\Gamma$ is the fundamental group of a 3-manifold:

\begin{proposition} \label{prop:SL4R 3mfld grp}
Suppose $(\Gamma,\peripherals)$ is a relatively hyperbolic group with $\partial(\Gamma,\peripherals)$ connected and not homeomorphic to $S^1$. Suppose $\rho\colon \Gamma \to \SL_4(\Rb)$ is 1-Anosov relative to $\peripherals$. Then $\Gamma / \ker \rho$ is the fundamental group of a compact 3-manifold $N$ \textup{(}possibly with boundary\textup{)}.

\begin{proof} Because $\partial(\Gamma,\peripherals)$ is connected and not homeomorphic to $S^1$, the image subgroup $\rho(\Gamma)$ preserves a domain $\Omega \subset \proj(\Rb^4)$.
Then $\Omega / \rho(\Gamma)$ is a 3-manifold $M$, and $\rho(\Gamma) \cong \pi_1(M)$.

We may assume $M$ is compact by the Scott core theorem \cite{Scott}. This states that: if $\Gamma$ is the fundamental group of a 3-manifold $M$ (possibly with boundary) and $\Gamma$ is finitely-generated, then there exists a compact three-dimensional submanifold $N \subset M$, called the compact core, such that the inclusion $\iota\colon N \into M$ induces an isomorphism on fundamental groups.
\end{proof}
\end{proposition}

\begin{proof}[Proof of Theorem~\ref{thm:SL4R connected}]
If $\xi_{\rho}^1\colon\partial(\Gamma,\mathcal{P})\rightarrow \mathbb{P}(\mathbb{R}^4)$ restricted to $\Lambda_{\Gamma_0}\subset \partial(\Gamma,\peripherals)$ is constant, then $\Gamma_0$ has to be a virtually nilpotent group by Proposition \ref{prop:peripheral unique flags} and Proposition \ref{prop:relA peripherals are v nilpotent}. Suppose this is not the case, and assume that $\Gamma_0$ is not virtually a free group or surface group.
Then, by Proposition~\ref{prop:SL4R 3mfld grp}, we know that $\Gamma_0 / \ker \rho$ is the fundamental group of a compact 3-manifold, and, by Proposition~\ref{prop:connected bowditch boundaries for 1-Anosov in SL4R}, we know that each of the peripheral subgroups of $\Gamma_0 / \ker \rho$ is virtually free abelian of rank at most 2. Hence (iii) holds. 
\end{proof}

\begin{remark}
If $\Gamma_0 \cong \pi_1 N^3$ where $N^3$ is a compact 3-manifold admitting a JSJ decomposition (this happens e.g.\ when $N^3$ has empty or toroidal boundary, see \cite[Thm.\ 7.2.2]{AFW}),
then we can further conclude, using the JSJ decomposition, that $N^3$ consists of hyperbolizable pieces joined across tori.

We remark that the JSJ decomposition represents $\Gamma_0$ as a graph of groups with vertex groups the fundamental groups of the JSJ components, and edge groups the fundamental groups of the JSJ tori. This does \emph{not} coincide with the JSJ splitting of $\Gamma_0$ as a relatively hyperbolic group (with connected Bowditch boundary) over elementary subgroups. 

More generally, e.g.\ if $\Gamma_0 \cong \pi_1 N^3$ where $N^3$ is a compact 3-manifold with atoroidal boundary, it is less clear how we can use 3-manifold theory to say more about what $N^3$ can be.
\end{remark}

\begin{proof}[Proof of Theorem~\ref{thm:SL4R}] By Theorem~\ref{thm:splittings from disconnected Bowditch boundary}, there exists a maximal splitting $\Gamma=\pi_1(\mathcal{G})$ of $\Gamma$ relative to $\mathcal{P}$, where $\mathcal{G}$ is a finite graph of groups whose edge groups are finite and whose vertex groups are relatively hyperbolic with connected Bowditch boundary and connected limit set in $\partial(\Gamma, \mathcal{P})$.
In particular, each of the vertex groups either (i) has singleton boundary or (ii) has connected (non-singleton) boundary.

The vertex groups in case (i), with singleton boundary, must be virtually peripheral, and hence are virtually nilpotent.

If there is a vertex group $\Gamma_0$ which falls into case (ii), and is not virtually a free group or surface group, then, by Theorem~\ref{thm:SL4R connected}, $\Gamma_0$ is a finite extension (by $\ker \rho$) of a group which is virtually the fundamental group of a compact 3-manifold and hyperbolic relative to virtually free abelian subgroups of rank at most 2.

An application of Proposition~\ref{graph-torsionfree} then yields the desired conclusion.
\end{proof}

\section{General constraints on the Bowditch boundary}

\subsection{Restrictions on topological dimension of the boundary} \label{subsec:restrictions on dim bdy}

In general, if a group $\Gamma$ admits a $k$-Anosov representation into $\SL_d(\Rb)$, where $1\leq k< \frac{d}{2}$, then the topological dimension of the boundary $\partial_\infty\Gamma$ is at most $d-k-1$ \cite[Thm.\ 1.3]{CanaryTsouvalas}. The same argument in fact works for relatively $k$-Anosov representations.

\begin{theorem}\label{upperbound-dim} Let $d\geq 2$ an integer, $1\leq k\leq \frac{d}{2}$ and $\rho:(\Gamma,\mathcal{P})\rightarrow \mathsf{SL}_d(\mathbb{R})$ be a relatively $k$-Anosov representation. Then we have the upper bound $$\textup{dim}(\partial(\Gamma,\mathcal{P}))\leq d-k-1,$$ unless $(d,k)\in \{(2,1),(4,2),(8,4),(16,8)\}$ and $\partial(\Gamma,\mathcal{P})\cong S^{d-k}$.
\end{theorem}

Before we give the proof of Theorem \ref{upperbound-dim} we need the following proposition. 

\begin{proposition}\label{KapovichBenakli-rel}
If $(\Gamma,\Pc)$ admits a relatively $1$-Anosov representation into $\SL_d(\Rb)$ with $d \geq 4$, and $\partial(\Gamma,\Pc)$ has dimension $d-2$, then $\partial(\Gamma,\Pc) \cong S^{d-2}$.
\begin{proof} This follows from \cite[Thm.\ 4.4]{KapoBen}. The result is stated there for Gromov boundaries of hyperbolic groups, but also applies to the Bowditch boundary of a non-elementary relatively hyperbolic group $(\Gamma,\peripherals)$ with the same proof, since pairs $(\gamma^+, \gamma^-)$ of fixed points of loxodromic elements in any relatively hyperbolic group $(\Gamma,\peripherals)$ are dense in $(\partial(\Gamma,\peripherals))^2$.
\end{proof} \end{proposition}

\begin{proof}[Proof of Theorem \ref{upperbound-dim}] We work as in the proof of \cite{CanaryTsouvalas}. Let us fix $x_0\in \partial(\Gamma, \mathcal{P})$, a vector $v\in \mathbb{R}^d \smallsetminus \xi_{\rho}^{d-k}(x_0)$ and consider the map $$f_{v}:\partial(\Gamma, \mathcal{P})\smallsetminus \{x_0\}\rightarrow \mathbb{P}\big(\xi_{\rho}^{d-k}(x_0)\oplus \mathbb{R}v\big) \smallsetminus \mathbb{P}(\xi_{\rho}^{d-k}(x_0))\cong \mathbb{R}^{d-k}$$ defined as follows: $$f_{v}(x):=\big[\xi_{\rho}^{k}(x)\cap \big(\xi_{\rho}^{d-k}(x_0)\oplus \mathbb{R}v\big)\big].$$ Since $(\xi_{\rho}^{k},\xi_{\rho}^{d-k})$ are transverse, the map $f_{v}$ is well-defined, continuous, proper and injective. In particular, we immediately conclude that $\textup{dim}(\partial(\Gamma,\mathcal{P}))\leq d-k$.

Now let $\mathcal{E}_{\rho,k}:=\bigcup_{\eta\in \partial(\Gamma, \mathcal{P})}\mathbb{P}(\xi_{\rho}^{k}(\eta))$. First, observe that since $(\xi_{\rho}^{k},\xi_{\rho}^{d-k})$ are transverse and compatible, there is a fibration $\pi\colon \mathcal{E}_{\rho,k} \rightarrow \partial(\Gamma,\mathcal{P})$ with fibers are homeomorphic to $\mathbb{P}(\mathbb{R}^k)$. More explicitly, for a line $[u]\in \mathcal{E}_{\rho,k}$ $\pi([u])\in \partial(\Gamma, \mathcal{P})$ is the unique point $\eta_{u}\in \partial(\Gamma,\mathcal{P})$ such that $[u]\in \mathbb{P}(\xi_{\rho}^{k}(\eta_v))$.

In particular, if $\textup{dim}(\partial(\Gamma, \mathcal{P}))=d-k$, then, arguing  as in the proof of \cite[Thm.\ 1.3]{CanaryTsouvalas}, we conclude that the image of the map $f_{v}$ has to contain an open subset of the affine chart $\mathbb{P}\big(\xi_{\rho}^{d-k}(x_0)\oplus \mathbb{R}v\big) \smallsetminus \mathbb{P}(\xi_{\rho}^{d-k}(x_0))$ and thus, by Proposition \ref{KapovichBenakli-rel}, we have $\partial(\Gamma,\mathcal{P})\cong S^{d-k}$. In particular, $f_{v}$ is onto and therefore $\mathcal{E}_{\rho,k}$ contains $\mathbb{P}(\mathbb{R}^d)\smallsetminus \mathbb{P}(\xi_{\rho}^{d-k}(x_0))$. Thus, we have $\mathcal{E}_{\rho,k}=\mathbb{P}(\mathbb{R}^d)$, so we obtain a fibration $\mathbb{P}(\mathbb{R}^k)\xhookrightarrow{} \mathbb{P}(\mathbb{R}^d)\twoheadrightarrow S^{d-k}$. By the classification of such fibrations by Adams \cite{Adams}, we deduce that $(d,k)\in \{(2,1),(4,2),(8,4),(16,8)\}$.\end{proof}

\subsection{Excluding spheres in the boundary} \label{subsec:excluding spheres}
The upper bound for the cohomological dimension of the domain group of an Anosov representation provided by \cite[Thm.\ 1.3]{CanaryTsouvalas} is known to be achieved only in a few cases. The exceptional cases where the bound of Proposition \ref{upperbound-dim} is achieved are when either $k\in \{1,2,4\}$, see also the discussion in \cite[\S 10]{CanaryTsouvalas}. The case where $k=1$ is achieved when $\Gamma$ is a lattice in $\mathsf{SO}(d,1)$. The case where $k=2$ is achieved is when $\Gamma$ is a lattice in $\mathsf{SU}(d,1)$ and we restrict the natural embedding $\mathsf{SL}_{d+1}(\mathbb{C})\xhookrightarrow{} \mathsf{SL}_{2d+2}(\mathbb{R})$ on $\Gamma<\mathsf{SU}(d,1)$. Finally, the case where $k=4$ is achieved when $\Gamma$ is a lattice in $\Sp(d,1)$ and we restrict the natural embedding $\mathsf{SL}_{d+1}(\mathbb{H})\xhookrightarrow{} \mathsf{SL}_{4d+4}(\mathbb{R})$ on $\Gamma<\Sp(d,1)$. 

In each of these exceptional cases, we see that the Bowditch boundary is a $(d-k)$-sphere. More generally, in the non-exceptional cases, we see that the Bowditch boundary has dimension at most $d-k-1$. It is then natural to ask if the boundary embeds into a sphere of dimension $d-k-1$. Note that when $k=1$ and $\partial(\Gamma,\Pc)$ is connected and not a circle, Proposition \ref{prop:rel Anosov preserving domain} gives a positive answer. More generally, this question remains open, even in the non-relative case (see Question \ref{embedding-spheres}).

Theorem~\ref{thm:excluding spheres} offers some further evidence for a positive answer to this question. 
% \begin{theorem} \label{thm:excluding spheres}
% Let $d,k\in \Nb$ with $d\geq 2k+1$ and $\mathbb{K}=\mathbb{R}$ or $\mathbb{C}$. Let $(\Gamma,\peripherals)$ be a relatively hyperbolic group whose Bowditch boundary $\partial(\Gamma,\peripherals)$ properly contains a sphere of dimension $r_{\Kb}(d-k)-1$ \textup{(}where $r_{\Rb}=1$ and $r_{\Cb}=2$\textup{)}. Then there is no relatively $k$-Anosov representation $\rho \colon(\Gamma,\peripherals) \to \SL_d(\Kb)$.
% \end{theorem}
It will be a consequence of the following more general theorem, since the limit maps associated to a relatively Anosov representation yield a transverse subset. We remark that the more general theorem also applies to e.g.\ $\Psf_k$-transverse subgroups of $\SL_d(\Kb)$, as defined by Canary--Zhang--Zimmer~\cite{CZZ}.

\begin{theorem} \label{thm:excluding spheres general}
Let $d,k\in \Nb$ with $d\geq 2k+1$ and $\mathbb{K}=\mathbb{R}$ or $\mathbb{C}$. 
There is no transverse set 
$$
\Lambda \subset \Gr_k(\Kb^d) \times \Gr_{d-k}(\Kb^d)
$$
which properly contains a sphere of dimension $r_{\Kb}(d-k)-1$ \textup{(}where $r_{\Rb}=1$ and $r_{\Cb}=2$\textup{)}. 
\end{theorem}

We first state some notation we will use below. 
We denote by $$S^{r}:=\big\{(x_1,\dots,x_r)\in \Rb^{r+1}: x_1^2+x_2^2+\cdots+x_{r+1}^2=1 \big\}$$ the standard Euclidean sphere of dimension $r\geq 1$. Given a point $x\in S^r$,  $-x\in S^r$ denotes the antipodal point of $x$. 
We fix the standard inner product on $\Rb^d$. For a subspace $V\subset \Rb^{r+1}$, $V^{\perp}$ is the orthogonal complement of $V$. 
Let $d \geqslant 2$ and $(e_1,\ldots,e_d)$ be the canonical basis of $\mathbb{R}^d$. The Pl${\textup{\"u}}$cker embeddings $$\tau_{k}^{+}:\mathsf{Gr}_{k}(\mathbb{R}^d) \rightarrow \mathbb{P}(\wedge^k \mathbb{R}^d) \quad\mbox{and}\quad \tau_{k}^{-}:\mathsf{Gr}_{d-k}(\mathbb{R}^d) \rightarrow \mathsf{Gr}_{d_k-1}(\wedge^{k}\mathbb{R}^d),$$ where $d_k=\binom{d}{k}$, are the maps defined as follows: for any $g\in \mathsf{GL}_d(\mathbb{R})$, 
$$\tau_{k}^{+}\big(g\langle e_1,\ldots,e_k\rangle\big)=\big[ge_1\wedge \ldots \wedge ge_k\big] \quad\mbox{and}\quad \tau^{-}_{k}\big(g\langle e_{k+1},\ldots,e_d\rangle \big)=\big[(\wedge^{k}g)W_{k}\big],$$ where $W_k:=\big\langle \left\{e_{i_1}\wedge \ldots \wedge e_{i_{k}}: \{i_1,\ldots,i_{k}\} \neq \{1,\ldots,k\} \right\} \big\rangle$. The maps $\tau_k^{+}$ and $\tau_{k}^{-}$ are embeddings.

There will be several cases, depending on $d$ and $k$; in all cases, we will argue by contradiction. The following case is more straightforward, and we will reuse the main idea in its proof in later results.

\begin{lemma} \label{lem:thm5.3 case 3 general}
Suppose $d=2k+1$ and $k \in \mathbb{N}$ is odd. Suppose $\Lambda \subset \Gr_k(\Rb^d) \times \Gr_{d-k}(\Rb^d)$ is a transverse set. Then $\Lambda$ cannot properly contain a $k$-sphere. 
\end{lemma}
\begin{proof}
We argue by contradiction.
Suppose that there is a transverse set $\Lambda \subset \Gr_k(\Rb^d) \times \Gr_{d-k}(\Rb^d)$, an embedding $f \colon S^{k} \into \Lambda$ and $x_0\in \Lambda \smallsetminus f(S^{k})$.

Let $\pi_k \colon \Lambda \to \Gr_k(\Rb^d)$ and $\pi_{k+1} \colon \Lambda \to \Gr_{k+1}(\Rb^d)$ be the natural projections. By transversality, these are homeomorphisms.

Note that since $x_0\in \partial(\Gamma,P) \smallsetminus f(S^{k})$, by transversality, the map $\tau_{k}^{+}\circ \pi_k \circ f\colon S^{k}\to \proj(\wedge^{k}\Rb^d)$ avoids the hyperplane $\tau_{k}^{-}(\pi_{d-k}(x_0))\subset \wedge^k \mathbb{R}^{2k+1}$. Thus, we obtain a continuous lift $\widetilde{\xi}_k\colon S^{k}\to \mathbb{S}(\wedge^k \Rb^d)$ of $\tau_{k}^{+}\circ \pi_k \circ f$, where $\mathbb{S}(\cdot)$ denotes the unit sphere in the vector space. In addition, the continuous map $\mathsf{F}_{1}\colon S^{k}\to \proj(\pi_{k+1}(x_0))\smallsetminus \proj(\pi_k(x_0))$ defined by $$\mathsf{F}_1(x)=\big[\pi_{k+1}(f(x))\cap \pi_{k+1}(x_0)\big]$$ admits a continuous lift, i.e.\ there is $v_0\in \pi_{k+1}(x_0)\smallsetminus \pi_{k}(x_0)$ and a continuous map $\mathsf{F}_1'\colon S^{k}\to \pi_{k}(x_0)\cong \Rb^k$ such that $[\mathsf{F}_1'(x)+v_0]=[\mathsf{F}_1(x)]$ for every $x\in S^{k}$.

By transversality for $x,y\in S^{k}$ distinct we have that $\pi_{k+1}(f(x))=\pi_{k}(f(x))\oplus \mathbb{R} \mathsf{F}_1(x)$ and $$\Rb^d=\pi_{k}(f(x))\oplus \pi_{k}(f(y))\oplus \mathbb{R} \mathsf{F}_1(x),$$ so we obtain a continuous, non-vanishing map $\mathcal{H}\colon (S^k)^{(2)}\to \Rb^{\ast}$ given by $$\mathcal{H}(x,y):=\widetilde{\xi}_k(x)\wedge \widetilde{\xi}_k(y) \wedge \big( \mathsf{F}_1'(x)+v_0\big).$$ 

Now $\mathsf{F}_1'$ is a continuous map from the $k$-sphere to an affine chart in projective $k$-space, and hence cannot be injective. Hence there are $x_1,y_1\in S^{k}$ distinct with $\mathsf{F}_1'(x_1)=\mathsf{F}_1'(y_1)$. But then, since $k\in \mathbb{N}$ is odd, we have \begin{align*}\mathcal{H}(y_1,x_1)&=\widetilde{\xi}_k(y_1)\wedge \widetilde{\xi}_k(x_1) \wedge \big( \mathsf{F}_1'(y_1)+v_0\big)\\
&=(-1)^{k^2}\widetilde{\xi}_k(x_1)\wedge \widetilde{\xi}_k(y_1) \wedge \big( \mathsf{F}_1'(x_1)+v_0\big)\\ &=-\mathcal{H}(x_1,y_1).\end{align*} This contradicts the fact that $\Hc$ is a non-zero continuous map. 
\end{proof}

\begin{proof}[Proof of Theorem \ref{thm:excluding spheres general}]
We will prove the statement when $\Kb = \Rb$. The case where $\mathbb{K}=\mathbb{C}$ then  follows by the real case, since this says there is no transverse subset of $\Gr_{2k}(\Rb^{2d}) \times \Gr_{2d-2k}(\Rb^{2d})$ properly containing a $(2d-2k-1)$-sphere, and $\Gr_k(\Cb^d) \subset \Gr_{2k}(\Rb^{2d})$.

We argue by contradiction.

Set $m=m(d,k):=d-k-1$. Suppose there is a transverse subset $\Lambda \subset \mathsf{Gr}_k(\mathbb{R}^d)\times \mathsf{Gr}_{d-k}(\mathbb{R}^d)$ and an embedding $f \colon S^{m} \into \Lambda$ and $x_0\in \Lambda \smallsetminus f(S^{m})$. 
If $d=2k+1$ and $k$ is odd, then we are done by Lemma~\ref{lem:thm5.3 case 3 general}.

The proof of the other cases will be more involved.

Let $W_0:=\pi_{d-k}(x_0)\cap (\pi_k(x_0))^\perp$. We have $\dim(W_0)=d-2k$. Let us also fix non-zero vectors $v_1\in W_0$ and $\omega_1\in \pi_{k}(x_0)$. Consider the  maps 
\begin{align*} 
\mathsf{F} &\colon S^{m}\to \proj\big((\pi_k(x_0))^\perp \oplus \Rb\omega_1\big)\smallsetminus \proj\big( (\pi_k(x_0))^\perp \big),\\ 
\mathsf{G} &\colon S^{m}\to \proj\big(\pi_{k}(x_0)\oplus \Rb v_1\big)\smallsetminus \proj(\pi_{k}(x_0)),\\ 
\mathsf{H} &\colon S^{m}\to \proj(\Rb^d)
\end{align*} 
defined as follows: for all $x \in S^m$,
\begin{align*}
\mathsf{F}(x)&:= (\pi_{d-k}(f(x)))^\perp \cap \big((\pi_k(x_0))^\perp \oplus \Rb\omega_1\big)\big], \\ \mathsf{G}(x)&:=\big[\pi_{d-k}(f(x)) \cap \big(\pi_{k}(x_0)\oplus \Rb v_1 \big)\big] ,\\ \mathsf{H}(x)&:=\big[\pi_{d-k}(f(x)) \cap \big(\pi_{k}(f(-x))\oplus \mathsf{G}(-x)\big)\big].
\end{align*}

Note that $\mathsf{F} \colon S^m \rightarrow \mathbb{P}(\xi_{\rho^{\ast}}^{d-k}(x_0)\oplus \mathbb{R}\omega_1)$ and $\mathsf{G} \colon S^m \to \mathbb{P}(\pi_k(x_0)\oplus \mathbb{R}v_1)$ are well-defined and continuous by transversality. In addition, $\mathsf{H}:S^m\rightarrow \mathbb{P}(\mathbb{R}^d)$ is well-defined and continuous since $\pi_{d-k}(f(x))\oplus \pi_{k}(f(-x))=\Rb^d$ for every $x\in S^{m}$ and $\mathsf{G}(-x)\in \pi_{k}(x_0)\oplus \Rb v_1\subset \pi_{d-k}(x_0)$ is transverse to $\pi_{k}(f(-x))$ since $x_0\in \partial(\Gamma, \mathcal{P})\smallsetminus f(S^{m})$. Observe also that there are continuous functions $\mathsf{F}_{0}\colon S^{m}\to (\pi_k(x_0))^\perp$ and $\mathsf{G}_0\colon S^{m}\to \pi_{k}(x_0)$ such that \begin{align*}\widetilde{\mathsf{F}}(x)&=\mathsf{F}_0(x)+\omega_1, \ x\in S^{m}\\ \widetilde{\mathsf{G}}(x)&=\mathsf{G}_0(x)+v_1, \ x\in S^{m}\end{align*} are continuous lifts of $\mathsf{F}$ and $\mathsf{G}$ respectively. 

Now we are going to construct a lift for the map $\mathsf{H}$. First, note that for every $x\in S^{m}$ a lift of $\mathsf{H}(x)\in \mathbb{P}(\pi_{d-k}(f(x)))$ can be chosen to be of the form $\mathsf{R}(x)+\widetilde{\mathsf{G}}(-x)$ for some $\mathsf{R}(x)\in \pi_{k}(f(-x))$, since $\pi_{k}(f(-x))\cap \pi_{d-k}(f(x))=\{\bf 0\}$. Thus, there are maps $\widetilde{\mathsf{H}}\colon S^{m}\to \Rb^d\smallsetminus \{\bf 0\}$ and $\mathsf{R}\colon S^{m}\to \Rb^d$ with the property that $$\widetilde{\mathsf{H}}(x)=\mathsf{R}(x)+ \mathsf{G}_0(-x)+v_1,\ x\in S^{m}$$ is a lift of $\mathsf{H}$, $\widetilde{\mathsf{H}}(x)\in \pi_{d-k}(f(x))$ and $\mathsf{R}(x)\in \pi_{k}(f(-x))$ for every $x\in S^{m}$.
\medskip

\noindent {\em Claim: The maps $\widetilde{\mathsf{H}}$ and $\mathsf{R}$ are well-defined and continuous.} 
%We first observe that \hbox{$\mathsf{R}(x)\in \pi_{k}(f(-x))$} is the unique vector such that $\mathsf{R}(x)+\mathsf{G}_0(-x)+v_1$ is a lift of $\mathsf{H}(x) \in \mathbb{P}(\pi_{d-k}(f(x)))$. Indeed, if $\mathsf{R}'(x)\in \pi_{k}(f(-x))$ is another vector such that $[\mathsf{R}'(x)+G_0(-x)+v_1]=\mathsf{H}(x)$,
If $\mathsf{R}'(x),\mathsf{R}'(x)\in \pi_{k}(f(-x))$ are two vectors such that $[\mathsf{R}(x)+\mathsf{G}_0(-x)+v_1] = [\mathsf{R}'(x)+\mathsf{G}_0(-x)+v_1]=\mathsf{H}(x)$,
then $$\mathsf{R}(x)-\mathsf{R}'(x)\in \pi_{k}(f(-x))\cap \pi_{d-k}(f(x))=\{\bf 0\}$$ and hence $\mathsf{R}(x)=\mathsf{R}'(x)$. It follows that $\mathsf{R},\widetilde{\mathsf{H}}:S^m \rightarrow \mathbb{R}^d$ are both well-defined.

Now in order to prove that $\mathsf{R}\colon S^m \rightarrow \mathbb{R}^d$ is continuous we first observe that the map $x\mapsto \left\|\mathsf{R}(x)\right\|$ is bounded. If not, there is a sequence $(y_n)_{n\in \mathbb{N}}\subset S^{m}$ with $\lim_n y_n=y_0$ and $\lim_n ||\mathsf{R}(y_n)||=\infty$. Since $x\mapsto \left\|{\mathsf{G}}_0(x)+v_1\right\|$ is continuous and bounded, up to a passing to a subsequence, we have \begin{equation}\label{equalitylim} \lim_{n\rightarrow \infty} \frac{\mathsf{R}(y_n)}{||\mathsf{R}(y_n)||}=\lim_{n \rightarrow \infty} \frac{\widetilde{\mathsf{H}}(y_n)}{||\mathsf{R}(y_n)||},\end{equation} where both limits exist and are non-zero vectors of $S^{d-1}$. Moreover, since $\frac{1}{||\mathsf{R}(y_n)||}\mathsf{R}(y_n) \in \pi_{k}(f(-y_n))$ for every $n\in \mathbb{N}$, we have $\lim_n \frac{1}{||\mathsf{R}(y_n)||}\mathsf{R}(y_n) \in \mathbb{S}(\pi_{k}(f(-y_0)))$. Similarly, $\frac{1}{||\mathsf{R}(y_n)||}\widetilde{\mathsf{H}}(y_n) \in \pi_{d-k}(f(y_n))$ for every $n$, so $\lim_n \frac{1}{||\mathsf{R}(y_n)||}\widetilde{\mathsf{H}}(y_n) \in \mathbb{S}(\pi_{d-k}(f(y_0)))$. However, by (\ref{equalitylim}), this contradicts that by transversality we have $\pi_{k}(f(-y_0))\cap \pi_{d-k}(f(y_0))=\{\bf 0\}$. It follows that the map $x\mapsto \mathsf{R}(x)$ is bounded. 

 Now it is a formality to check that $\mathsf{R}$ and $\widetilde{\mathsf{H}}$ are continuous. Let $(z_n)_{n\in \mathbb{N}}\subset S^{m}$ be an arbitrary sequence with $\lim_n z_n=z_0$ and $\mathsf{R}_{\infty}\in \pi_{k}(f(-z_0))$ an accumulation point of $(\mathsf{R}(z_n))_{n\in \mathbb{N}}$. Then, note that $\mathsf{R}_{\infty}+\mathsf{G}_0(-z_0)+v_1\neq 0$ is a lift of $\mathsf{H}(z_0)\in \mathbb{P}(\pi_{d-k}(f(z_0)))$. Hence there is $\lambda \in \Rb^{\ast}$ such that $$\mathsf{R}_{\infty}+\mathsf{G}_0(-z_0)+v_1=\lambda \mathsf{R}(z_0)+\lambda(\mathsf{G}_0(-z_0)+v_1).$$ Now note that $\mathsf{R}_{\infty}\in \pi_{k}(f(-z_0))$, $\mathsf{R}(z_0)\in \pi_{k}(f(-z_0))$, $\mathsf{G}_0(-z_0)+v_1\in \pi_{k}(x_0)\oplus \Rb v_1$ and $$\pi_{k}(f(-z_0))\cap (\pi_{k}(x_0)\oplus \Rb v_1)\subset \pi_{k}(f(-z_0))\cap \pi_{d-k}(x_0)=\{\bf 0\},$$ so we have $\lambda=1$ and $\mathsf{R}_{\infty}=\mathsf{R}(z_0)$. Finally, the claim follows and $\widetilde{\mathsf{H}}$ and $\mathsf{R}$ are well-defined and continuous. \qed \medskip

Now by using the orthogonal decomposition 
\begin{equation} \label{perp1} 
\Rb^d=\pi_{k}(x_0)\oplus (\pi_k(x_0))^\perp
\end{equation} we may write, for any $x \in S^m$, \begin{equation} \label{perpR} 
\mathsf{R}(x)=\underbrace{\mathsf{R}_1(x)}_{\in (\pi_k(x_0))^\perp}+\underbrace{\mathsf{R}_2(x)}_{\in \pi_{k}(x_0)}
\end{equation}  where $\mathsf{R}_{1} \colon S^{m} \to (\pi_k(x_0))^\perp$ and $\mathsf{R}_2\colon S^{m} \to \pi_{k}(x_0)$ are continuous functions. In particular, we obtain the decomposition for the map $\widetilde{\mathsf{H}}$:
\begin{equation} \label{perpH} \widetilde{\mathsf{H}}(x)=\mathsf{R}(x)+\underbrace{\mathsf{G}_0(-x)}_{\in \pi_{k}(x_0)}+\underbrace{v_1}_{\in (\pi_k(x_0))^\perp}=\underbrace{\mathsf{R}_1(x)+v_1}_{\in (\pi_k(x_0))^\perp}+ \underbrace{\mathsf{R}_2(x)+\mathsf{G}_0(-x)}_{\in \pi_{k}(x_0)} \quad \forall\ x\in S^{m}.\end{equation}
Also recall that \begin{equation}\label{perpF} \widetilde{\mathsf{F}}(x)=\underbrace{\mathsf{F}_0(x)}_{\in (\pi_k(x_0))^\perp}+ \underbrace{\omega_1}_{\in \pi_{k}(x_0)} \quad\forall \ x\in S^{m}.\end{equation} 

Since $\widetilde{\mathsf{H}}(x)\in \pi_{d-k}(f(x))\cap \pi_{d-k}(f(-x))$, we have that the vector $\widetilde{\mathsf{F}}(\pm x)\in \mathbb{R}^d$ is perpendicular to $\widetilde{\mathsf{H}}(x)\in \mathbb{R}^d$ for every $x\in S^m$. In other words, by (\ref{perp1}) and (\ref{perpH}), (\ref{perpF}) we have that 
$$\big \langle \mathsf{R}_1(x)+v_1, \mathsf{F}_0(x)\big \rangle+ \big \langle \mathsf{R}_2(x)+\mathsf{G}_0(-x), \omega_1\big \rangle=0, $$ $$\big \langle \mathsf{R}_1(x)+v_1, \mathsf{F}_0(-x)\big \rangle+ \big \langle \mathsf{R}_2(x)+\mathsf{G}_0(-x), \omega_1\big \rangle=0 $$ for all $x\in S^{m}$ and in particular  $$\big \langle \mathsf{R}_1(x)+v_1, \mathsf{F}_0(x)-\mathsf{F}_0(-x)\big \rangle=0, \ x\in S^{m}.$$ 

We remark that, by (\ref{perpH}), $\mathsf{R}_1(x)+v_1\neq0$ since $\widetilde{\mathsf{H}}(x)\in \pi_{d-k}(f(x))\subset \mathbb{R}^d\smallsetminus \pi_{k}(x_0)$. Also, the map $\mathsf{F}_0\colon S^{m}\to (\pi_k(x_0))^\perp$ is injective since $\pi_k$ and hence $\mathsf{F}$ is injective. Finally, we obtain the equality \begin{equation}\label{perpRF} \Bigg \langle \frac{\mathsf{R}_1(x)+v_1}{\left\|\mathsf{R}_1(x)+v_1\right\|}, \mathsf{F}_0'(x)\Bigg \rangle=0 \quad\forall x\in S^{m},  \end{equation} where $\mathsf{F}_0'\colon S^{m}\to \mathbb{S}((\pi_k(x_0))^\perp)$ is the continuous map $$\mathsf{F}_0'(x):=\frac{\mathsf{F}_0(x)-\mathsf{F}_0(-x)}{\left\|\mathsf{F}_0(x)-\mathsf{F}_0(-x)\right\|}$$ which is odd and hence has non-zero degree. Let us also fix $\sigma\colon \mathbb{S}((\pi_k(x_0))^\perp)\to S^{m}$ a homeomorphism sending antipodal pairs to antipodal pairs, i.e. $\sigma(-u)=-\sigma(u)$ for every $u\in \mathbb{S}((\pi_k(x_0))^\perp)$. By using (\ref{perpRF}) we obtain the following continuous map $A\colon S^m \times [-\pi,\pi] \to S^{m}$ defined by $$A(x,t):=\sigma \Bigg((\sin t) \frac{\mathsf{R}_1(x)+v_1}{\left\|\mathsf{R}_1(x)+v_1\right\|}+(\cos t) \mathsf{F}_0'(x)\Bigg).$$ 
We note that $A$ gives us the following homotopies:
\medskip

\noindent \textup{(a)} $A\colon S^m \times \left[ 0, \frac\pi2 \right]$ is a homotopy between the odd map $\sigma \circ \mathsf{F}_0'\colon S^{m}\to S^{m}$ and the map
$\sigma \circ \Omega\colon S^{m}\to S^{m}$, where $$\Omega(x)= \frac{\mathsf{R}_1(x)+v_1}{\left\|\mathsf{R}_1(x)+v_1 \right\|}\in \mathbb{S}(\xi_{\rho^{\ast}}^{d-k}(x_0))\ \ x\in S^{m}.$$
\noindent \textup{(b)} $A\colon S^m \times [0,\pi]$ is a homotopy between $x\mapsto A(x,0)=(\sigma \circ \mathsf{F}_0')(x)$ and $x\mapsto A(x,\pi)=\sigma(-\mathsf{F}_0'(x))=(\theta \circ \sigma \circ \mathsf{F}_0')(x)$, for every $x\in S^m$, where $\theta\colon S^{m}\to S^{m}$ is the antipodal map.
\medskip

Now there are two cases to consider to finish the proof.
\medskip

\noindent {\bf Case 1.} {\em Suppose that $d\geq 2k+2$ (so $m\geq k+1$).} Note that $\textup{dim}(W_0)\geq 2$, so we may choose a unit vector $v_2\in W_0:=\pi_{d-k}(x_0)\cap (\pi_k(x_0))^\perp$ linearly independent from $v_1\in W_0$. By (a) we see that $\sigma \circ \Omega$ is homotopic to the odd map $\sigma \circ \mathsf{F}_0'$, so $\sigma \circ \Omega$ has non-zero degree and hence is a surjective onto $S^{m}$. In particular, $\Omega:S^m \rightarrow \mathbb{S}((\pi_k(x_0))^\perp)$ is surjective and there is $y_0\in S^{m}$ such that $$\mathsf{R}_1(y_0)=\lambda_0 v_2-v_1, \text{ where } \lambda_0=\left\|\mathsf{R}_1(y_0)+v_1\right\|.$$
By (\ref{perpR}), for $x:=y_0$, we have that $$\mathsf{R}(y_0)=\lambda_0 v_2-v_1+\mathsf{R}_2(y_0).$$ Note that $\mathsf{R}(y_0)\neq {\bf 0}$ since $\lambda_0 v_2-v_1\in  (\pi_k(x_0))^\perp \smallsetminus \{\bf 0\}$ and $\mathsf{R}_2(w_0)\in \pi_{k}(x_0)$. On the other hand, since $\lambda_0 v_2-v_1\in (\pi_k(x_0))^\perp \cap \pi_{d-k}(x_0)$, we have $\mathsf{R}(y_0)\in \pi_{d-k}(x_0)$. However, this contradicts the fact that $\mathsf{R}(x)\in \pi_{k}(f(-x))$ for every $x\in S^m$, and $\pi_{k}(f(-y_0)) \cap \pi_{d-k}(x_0)=\{\bf 0\}$.
\medskip

\noindent {\bf Case 2.} {\em Suppose that $d=2k+1$ (so $m=k$) and $k\in \mathbb{N}$ is even.} Recall that $\sigma\circ F_0'\colon S^{k}\to S^{k}$ is odd and has non-zero degree. By (b), since $\sigma \circ \mathsf{F}_0'$ and $\theta \circ \sigma  \circ \mathsf{F}_0'$ are homotopic and the antipodal map $\theta:S^k\rightarrow S^k$ has degree $-1$, we have that $\textup{deg}(\sigma \circ \mathsf{F}_0')=\textup{deg}(\theta \circ \sigma \circ F_0')=-\textup{deg}(\sigma \circ \mathsf{F}_0')$, so $\textup{deg}(\sigma \circ \mathsf{F}_0')=0$. This is a contradiction.
\medskip

In either case we obtain a contradiction, and hence the proof of the theorem is complete.\end{proof}

By adapting the argument in Lemma~\ref{lem:thm5.3 case 3 general},
% Case 3 of the proof of Theorem \ref{thm:excluding spheres}, 
we obtain a positive answer to Question \ref{embedding-spheres}
in the case of $k$-Anosov representations into $\mathsf{SL}_{2k+1}(\mathbb{R})$ when $k\in \mathbb{N}$ is odd and the Gromov boundary $\partial_{\infty}\Delta$ is connected.

\begin{theorem} \label{embedding-sphere1} Let $\Delta$ be an one-ended hyperbolic group and $k\in \mathbb{N}$ an odd integer. Suppose that there is a $k$-Anosov representation $\rho\colon \Delta \to \SL_{2k+1}(\Rb)$. Then $\partial_{\infty}\Delta$ embeds in $S^k$.\end{theorem}

Before we give the proof of Theorem \ref{embedding-sphere1} we will need the following lemma about cut pairs in the Gromov boundary of a one-ended hyperbolic group.

\begin{lemma}\label{dense-connected} Let $\Delta$ be an one-ended hyperbolic group whose Gromov boundary is not the circle. There exists a point $z\in \partial_{\infty}\Delta$ such that the set $$X(z):=\big\{w\in \partial_{\infty}\Delta: \partial_{\infty}\Delta \smallsetminus \{z,w\} \ \textup{is connected} \big\}$$ is a dense subset of $\partial_{\infty}\Delta$.\end{lemma}

\begin{proof} The group $\Delta$ is one-ended, hence its Gromov boundary $\partial_{\infty}\Delta$ is locally connected and has no global cut points by \cite{Bowditch-connected}. Since $\partial_{\infty}\Delta$ is not the circle, we may find distinct points $z,w\in \partial_{\infty}\Delta$, such that $\partial_{\infty}\Delta \smallsetminus \{z,w\}$ is connected. In particular, the set $X(z)$ is non-empty.
 
Now we prove that $X(z)$ intersects any open subset of $\partial_{\infty}\Delta$. Let $U\subset \partial_{\infty}\Delta$ be an open set. We may choose a sequence of elements $(\gamma_n)_{n\in \mathbb{Z}}$ of $\Delta$ such that $$\lim_{n\rightarrow \infty} \gamma_n=z, \lim_{n \rightarrow \infty} \gamma_{-n}=w, \lim_{n \rightarrow \infty} \gamma_n^{-1}=z^{-}, \lim_{n \rightarrow \infty} \gamma_{-n}^{-1}=w^{-}$$ and a finite subset $F\subset \Gamma$ such that $\gamma_{\ell+1}^{-1}\gamma_{\ell} \in F$ for every $\ell \in \mathbb{Z}$.

By the minimality of the action of $\Delta$ on $\partial_{\infty}\Delta$, may choose an infinite order element $\delta\in \Delta$ with $$\delta^{+}, \delta^{-}\in U\smallsetminus \{z,w, z^{-},w^{-}\}, \ g\delta^{+}\neq \delta^{\pm} \ \forall g\in F.$$ Then observe by the convergence group property that $\lim_n g\delta^{2n}w=g\delta^{+}$, $\lim_n \delta^{-n}g\delta^{2n}w=\delta^{-}$, hence we may choose $q>1$ large enough such that $\delta^{-q}g\delta^{2q}w \in U$ for every $g \in F$ and $\delta^qw, \delta^{2q} w, \delta^q z \in U\smallsetminus \{z^{-},w^{-}\}$.

Now we consider the sequence $(x_{n})_{n\in \mathbb{Z}}$ of points in $\partial_{\infty}\Delta$, defined as follows: $$x_{2k}:=\gamma_{k}\delta^q z,\ x_{2k+1}=\gamma_k \delta^{2q}w,\ k \in \mathbb{Z}.$$

Observe that since $\delta^qz, \delta^{2q}w\in U\smallsetminus \{z^{-},w^{-}\}$, we have $\lim_{n\rightarrow \infty}x_n=z$ and $\lim_{n \rightarrow -\infty} x_n=w$. Since $\partial_{\infty}\Delta\smallsetminus \{z,w\}$ is connected, \cite[Lem.\ 2.5]{Barrett} implies that there exists $p \in \mathbb{Z}$ such that $\partial_{\infty}\Delta \smallsetminus \{x_{p},x_{p+1}\}$ is connected.

 If $p\in \mathbb{Z}$ is even, $p=2r$ for some $r\in \mathbb{Z}$, then $\partial_{\infty}\Delta\smallsetminus \{\gamma_r \delta^q z, \gamma_{r}\delta^{2q}w\}\cong \partial_{\infty}\Delta\smallsetminus \{z, \delta^{q}w\}$ is connected. In particular, $\delta^q w\in X(z)$ and $\delta^q w\in U$, and hence $X(z)\cap U$ is non-empty.

 If $p\in \mathbb{Z}$ is odd, $p=2r+1$ for some $r\in \mathbb{Z}$, then $$\partial_{\infty}\Delta\smallsetminus \{\gamma_r \delta^{2q} w, \gamma_{r+1}\delta^{q}z\}\cong \partial_{\infty}\Delta\smallsetminus \{z, \delta^{-q}\gamma_{r+1}^{-1}\gamma_r \delta^{2q}w\}$$ is connected and $\delta^{-q}\gamma_{r+1}^{-1}\gamma_{r}\delta^{2q}w\in X(z)$. By the choice of $\delta\in \Delta$ and $p$ we have $\delta^{-q}\gamma_{r+1}^{-1}\gamma_{r}\delta^{2q}w\in U$ and hence $X(z)\cap U$ is non-empty.

This shows that the set $X(z)\subset \partial_{\infty}\Delta$ intersects every open subset of $ \partial_{\infty}\Delta$ and hence is dense.\end{proof}

\begin{proof}[Proof of Theorem \ref{embedding-sphere1}] If $\partial_{\infty}\Delta \cong S^1$ then the conclusion is immediate. So we may assume that $\partial_{\infty}\Delta$ is not the circle and by Lemma \ref{dense-connected}, there is a point $z\in \partial_{\infty}\Delta$ such that the set $X(z):=\{w\in \partial_{\infty}\Delta: \partial_{\infty}\Delta \smallsetminus \{z,w\} \ \textup{is connected}\}$ is dense in $\partial_{\infty}\Delta$.

Since $\rho$ is $k$-Anosov, by transversality, the map $f_{z}\colon \partial_{\infty}\Delta \smallsetminus \{z\} \rightarrow \mathbb{P}(\xi_{\rho}^{k+1}(z))\smallsetminus \mathbb{P}(\xi_{\rho}^k(z))$, $$f_{z}(x)=\big[\xi_{\rho}^{k+1}(x)\cap \xi_{\rho}^{k+1}(z)\big],\ x\neq z$$ is well-defined and proper. 

To prove the statement it is enough to check that the map $f_{z}$ is injective. Indeed, in this case $f_{z}$ will extend to an embedding between the one-point compactifications $f_{z}'\colon \partial_{\infty}\Delta \xhookrightarrow{} S^k$. 

By transversality, the image of the map $\tau_{k}^{+}\circ \xi_{\rho}^k:\partial_{\infty}\Delta \smallsetminus \{z\}\rightarrow \mathbb{P}(\wedge^k \mathbb{R}^{2k+1})$ avoids the hyperplane $\mathbb{P}(\tau_{k}^{-}(\xi_{\rho}^{k+1}(z)))$ and hence lifts to a continuous map $\widetilde{\xi}_{k}:\partial_{\infty}\Delta\smallsetminus \{z\}\rightarrow \mathbb{S}(\wedge^k \mathbb{R}^{2k+1})$. Thus, the map $\mathcal{H}_z\colon (\partial_{\infty}\Delta \smallsetminus \{z\})^{(2)}\rightarrow \mathbb{R}$ given by $$\mathcal{H}_z(x,y)=\widetilde{\xi}_k(x)\wedge \widetilde{\xi}_k(y)\wedge f_{z}(y),\ x\neq y,$$ is well-defined, continuous and non-zero by transversality. Now we prove the following claim.
\medskip

\noindent {\em Claim: $\mathcal{H}_z(x,y)\ \mathcal{H}_z(y,x)>0$ for every $x\neq y$.}

Let $x,y\in \partial_{\infty}\Delta \smallsetminus \{z\}$ and $U_x,U_y\subset \partial_{\infty}\Delta \smallsetminus \{z\}$ be arbitrary disjoint open neighbourhoods containing $x,y\neq z$ respectively. Since $\partial_{\infty}\Delta$ is locally connected \cite{Bowditch-connected}, we may assume that $U_x,U_y$ are connected. Since $X(z)$ is dense in $\partial_{\infty}\Delta$ we may choose $x'\in U_x \cap X(z)$, $x''\in U_x\cap X(z)$, $x'\neq x''$, and $y'\in U_y \cap X(z)$. Now observe that $(x,y)$ and $(x',y')$ lie in the connected subset $U_x\times U_y$ of $(\partial_{\infty}\Delta \smallsetminus \{z\})^{(2)}$. The pairs $(x',y')$ and $(x',x'')$, lie in the connected subset $\{x'\} \times (\partial_{\infty}\Delta \smallsetminus \{z,x'\})$, and $(x',x'')$ and $(y',x'')$ lie in the connected set $(\partial_{\infty}\Delta \smallsetminus \{z,x''\})\times \{x''\}$. 
Finally, $(y',x'')$ and $(y,x'')$ lie in the connected set $U_y \times \{x''\}$, and $(y,x'')$ and $(y,x)$ lie in the connected set $\{y\} \times U_x$.
We conclude that, since $\mathcal{H}_{z}$ is continuous, and $\mathcal{H}_z(x,y),\mathcal{H}_z(y,x)\neq 0$, we have $\mathcal{H}_z(x,y)\mathcal{H}_z(y,x)>0$ for every $x\neq y$. \qed \medskip

In particular, if $x,y\neq z$ and $f_{z}(x)=f_{z}(y)$ then necessarily $x=y$, otherwise, since $k\in \mathbb{N}$ is odd, we would have $\mathcal{H}_z(x,y)=-\mathcal{H}_z(y,x)$, which is impossible by the previous claim. We finally conclude that $f_z$ is injective and this finishes the proof of the theorem.\end{proof}

\section{Constraints in high cohomological dimension} \label{sec:high cd}
Recall that a representation $\rho\colon \Gamma \rightarrow \mathsf{SL}_d(\mathbb{R})$ is called a Benoist representation if $\rho$ has finite kernel and $\rho(\Gamma)$ preserves and acts properly discontinuously and cocompactly on a strictly convex domain of $\mathbb{P}(\mathbb{R}^d)$.

\begin{theorem} \label{thm:high cd}Let $d\geq 4$.
If $\rho\colon (\Gamma,\peripherals) \to \SL_d(\Rb)$ is a relatively 1-Anosov representation, and $\Gamma$ is a torsion-free group of cohomological dimension at least $d-1$, then $\peripherals = \varnothing$, $\Gamma$ is hyperbolic and $\rho$ is a Benoist representation.\end{theorem}

The proof of Theorem \ref{thm:high cd} follows the proof of \cite[Thm.\ 1.5]{CanaryTsouvalas}
\begin{proof} 
First, we assume that $\partial(\Gamma, \mathcal{P})$ is connected. Note that since $\Gamma$ cannot be a free group or surface group, $\partial(\Gamma, \mathcal{P})$ is not homeomorphic to $S^1$ by Theorem \ref{thm:convergence groups on S1 are Fuchsian}. Let $V:=\langle \xi_{\rho}^1(\partial(\Gamma,\mathcal{P}))\rangle$. By Propositions~\ref{prop:not circle implies 2-eps connected} and~\ref{prop:rel Anosov preserving domain}(1) the limit set $\xi_{\rho}^1(\partial(\Gamma, \mathcal{P}))$ is contained in an affine chart $A$ of $\mathbb{P}(V)$. 
By Proposition~\ref{prop:rel Anosov preserving domain}(3), $\rho(\Gamma)$ preserves a properly convex domain $\Omega_\rho \subset \proj(\Rb^d)$.
% Note that since $\Gamma$ acts as a convergence group on $\partial(\Gamma, \mathcal{P})$, the restriction $\rho|_{V}$ has to have finite kernel and discrete image by Observation~\ref{obs1}. In particular, the image $\rho|_{V}(\Gamma)$ preserves the convex hull $\mathcal{C}_{\rho}$ of $\xi_{\rho}^1(\partial(\Gamma, \mathcal{P}))$ in $A$ and its interior $\Omega_{\rho}:=\textup{Int}(\mathcal{C}_{\rho})$. 
By our assumption that $\textup{cd}(\Gamma)\geq d-1$, it follows that $\rho|_V(\Gamma)$ acts cocompactly on $\Omega_{\rho}$. Hence $\rho(\Gamma)$ cannot contain elements with all eigenvalues of modulus one~\cite[Prop.\ 10.3]{DGK} and $V=\mathbb{R}^d$. Since, for a relatively Anosov representation, all the eigenvalues of the image of a peripheral element have modulus one~\cite[Prop.\ 4.2]{ZZ1}, we conclude here we must have $\mathcal{P}=\varnothing$. Hence $\Gamma$ is word hyperbolic, $\rho$ is $1$-Anosov and (by cocompactness) $\partial\Omega_{\rho}=\xi_{\rho}^1(\partial_{\infty}\Gamma)$. Moreover, by Benoist's theorem \cite{BenoistI} the domain $\Omega_{\rho}$ is strictly convex.

It remains to rule out the case where $\partial(\Gamma,\mathcal{P})$ is disconnected. Indeed, if this is the case, since $\Gamma$ is torsion-free, by Theorem \ref{thm:splittings from disconnected Bowditch boundary}, there is a splitting $\Gamma=\Gamma_1 \ast \cdots \ast \Gamma_r$ and finite families of subgroups $\mathcal{P}_i$ of $\Gamma_i$ such that $\partial(\Gamma_i, \mathcal{P}_i)$ is connected for every $1\leq i \leq r$. Since $\textup{cd}(\Gamma)=\max_i \textup{cd}(\Gamma_i)$, we may assume that $\textup{cd}(\Gamma_1)\geq d-1$. By the previous case, we deduce that $\mathcal{P}_1=\varnothing$ and there is a properly convex domain $\Omega_1\subset \mathbb{P}(\mathbb{R}^d)$ 
such that $\rho(\Gamma_1)$ preserves and acts cocompactly on $\Omega_1$. Then $\rho|_{\Gamma_1}$ is $1$-Anosov and $$\Omega_1:=\mathbb{P}(\mathbb{R}^d)\smallsetminus \bigcup_{\eta \in \partial_{\infty}\Gamma_1}\mathbb{P}(\xi_{\rho}^{d-1}(\eta)).$$ In particular, by the transversality of the limit maps $\xi_{\rho}^1$ and $\xi_{\rho}^{d-1}$, for every $2\leq i \leq r$, $\xi_{\rho}^1\left(\partial(\Gamma_i,\mathcal{P}_i)\right)$ is contained in $\Omega_1$. This then implies: if $x\in \partial(\Gamma_i,\mathcal{P}_i)$ where $i \neq 1$, then the projective hyperplane $\mathbb{P}(\xi_{\rho}^{d-1}(x))$ has to intersect $\partial\Omega_1=\xi_{\rho}^1(\partial_{\infty} \Gamma_1)$; but this contradicts transversality since $\partial(\Gamma_i,\mathcal{P}_i)$ and $\partial_{\infty}\Gamma_1$ are disjoint.

Therefore, $\mathcal{P}=\varnothing$, $\Gamma$ is word hyperbolic group and $\rho(\Gamma)$ and $\rho$ is a Benoist representation.  \end{proof}

\section{Questions}

We know from the above that if $(\Gamma,\peripherals)$ is a relatively hyperbolic group and $\rho\colon(\Gamma,\peripherals)\to\mathsf{SL}_d(\mathbb{R})$ is a relatively $k$-Anosov representation with $d \geq 2k$, then the dimension of $\partial(\Gamma,\peripherals)$ is bounded above by $d-k-1$ (see also \cite{CanaryTsouvalas} for the non-relative case). We also know that, when $k=1$ and $\partial(\Gamma,\peripherals)$ is connected and not a circle, the image of $\rho$ preserves a domain in $\proj(\Rb^d)$ and hence $\partial(\Gamma,\peripherals)$ embeds into a $(d-2)$-sphere. 
As noted in \S\ref{subsec:excluding spheres} above, a natural question that then arises is the following:

\begin{question}\label{embedding-spheres} Suppose $\rho\colon (\Gamma,\mathcal{P})\rightarrow \mathsf{SL}_d(\mathbb{R})$ is a relatively $k$-Anosov representation and $d\geq 2k+1$. Does there exist an embedding of $\partial(\Gamma,\mathcal{P})$ into the sphere of dimension $d-k-1$?\end{question} 

More generally, it is known that groups admitting relatively Anosov representations must be relatively hyperbolic with virtually nilpotent peripheral subgroups (see Proposition~\ref{prop:relA peripherals are v nilpotent}), and must be finite extensions of linear groups, but no other restrictions are known (in the absence of bounds on the dimension of the target Lie group or of the domain group). This is similar to the situation in the non-relative case, where it remains possible that any linear hyperbolic group admits an Anosov representation.

\begin{question}
Do all linear relatively hyperbolic groups with virtually nilpotent peripheral subgroups admit relatively Anosov representations?
\end{question}

Beyond relatively Anosov representations, a larger, more flexible class of representations generalizing geometric finiteness in the setting of higher-rank Lie groups is given by the extended geometrically finite (EGF) representations introduced by Weisman \cite{Teddy_EGF}. Roughly speaking, these are defined in terms of extended boundary maps with good dynamical properties. In general, these maps are no longer homeomorphisms, but rather equivariant surjections from subsets of the flag space to the Bowditch boundary of the domain group. Our results here would not apply directly, but it is possible that some of the techniques can be adapted to the study of EGF representations. In particular, we can ask some of the following questions:
\begin{question}
Do all linear relatively hyperbolic groups admit EGF representations?
\end{question}

\begin{question}
If $(\Gamma,\Pc)$ is a relatively hyperbolic group with virtually nilpotent peripherals, and $\rho\colon (\Gamma,\Pc) \to \SL_3(\Rb)$ is an EGF representation, is $\Gamma$ necessarily a surface group or free group?

What happens without the restriction on the peripherals?
\end{question}

% \AtNextBibliography{\small}
\printbibliography
\end{document}